\documentclass[12pt]{article}
\usepackage{amsthm}
\usepackage{amsmath,bm}
\usepackage{enumerate}
\usepackage{amssymb}
\usepackage{latexsym}
\usepackage{amsfonts}
\usepackage{color}
\usepackage{secdot}
\usepackage{mathrsfs}
\usepackage{subfigure}
\usepackage{epsfig}
\usepackage{natbib}
\newtheorem{theorem}{Theorem}[section]
\newtheorem{counterexample}{Counterexample}[section]
\newtheorem{example}{Example}[section]
\newtheorem{corollary}{Corollary}[section]
\newtheorem{definition}{Definition}[section]
\newtheorem{lemma}{Lemma}[section]
\newtheorem{remark}{Remark}[section]

\begin{document}
\title{On comparison of the second-order statistics from independent
and interdependent exponentiated location-scale distributed random variables}
\author{{\large { Sangita {\bf Das}\thanks {Email address:
                sangitadas118@gmail.com}~ and Suchandan {\bf
                Kayal}\thanks {Email address (corresponding author):
                kayals@nitrkl.ac.in,~suchandan.kayal@gmail.com}}} \\
    { \em \small {\it Department of Mathematics, National Institute of
            Technology Rourkela, Rourkela-769008, India.}}}
\date{}
\maketitle
\begin{center}Abstract
\end{center}
Consider two batches of independent or interdependent  exponentiated location-scale distributed heterogeneous random variables. This article investigates ordering results for the second-order statistics from these batches when a vector of parameters is switched to another vector of parameters in the specified model. Sufficient conditions for the usual stochastic order and the hazard rate order are derived. Some applications of the established results are presented.

\section{Introduction}
Let $X_{1},\cdots,X_{n}$ be a random sample. The sample values can be arranged in ascending order, denoted by
$X_{1:n}\le \cdots\le X_{n:n}$. This increasing arrangement is known as the order statistics of the sample. The order statistics are of interest in many areas of statistics and applied probability including actuarial science, reliability, risk management and auction theory. Among various reliability systems, a popular system is the $k$-out-of-$n$ system. It functions if at least $k$ components operate. The order statistic $X_{n-k+1}$ represents the lifetime of the $k$-out-of-$n$ system. We refer to \cite{arnold1992first} and \cite{david2004order} for detailed description and applications of the order statistics. Note that in the context of the reliability and life testing experiments, the stochastic comparisons of the systems' lifetimes is of great importance. In this regard, various stochastic orders (see \cite{shaked2007stochastic}) play a vital role. In most of the cases, attention has been paid to establish stochastic comparisons of the lifetimes of series and parallel systems. The components' lifetimes may be independent or interdependent. Here, we present few recent developments on this topic. \cite{kochar2015stochastic} considered scale models and obtained likelihood ratio ordering between the largest order statistics. \cite{torrado2015magnitude} developed different ordering results to compare smallest order statistics arising from the proportional reversed hazard rate model. \cite{li2015ordering} presented various stochastic orderings between the sample maximums arising from two groups of proportional hazard dependent random variables. \cite{Li2016} studied stochastic comparisons of the order statistics from random variables following the scale
model. Dependent case was considered to obtain the usual stochastic order of the sample extremes.  \cite{hazra2017stochastic} considered the problem of stochastic comparison for the largest order statistics from location-scale family of distributions. \cite{hazra2018stochastic} obtained various stochastic comparison results between minimum order statistics for general location-scale model. \cite{kayal2019dls} and \cite{kayal2019ls} respectively obtained some ordering results for the extreme order
statistics arising from dependent and independent heterogeneous exponentiated location-scale (ELS) random
observations. \cite{zhang2019stochastic} compared parallel and series systems with heterogeneous resilience-scaled components. \cite{kundu2020stochastic} studied series systems with heterogeneous dependent and independent location-scale family distributed components and proposed some ordering results.
 \cite{li2020prh} established several comparison results for the largest order
statistics in accordance to the reversed hazard rate and likelihood ratio orders for the proportional reversed hazard rate model.

It is well-known that the second order statistic characterizes the time to failure of an $(n-1)$-out-of-$n$ system. In reliability and industrial engineering, this particular system is known as the fail-safe system (see \cite{barlow1996mathematical}). Further, for the bid in second-price reverse auction, the second-order statistic represents the winner's price (see \cite{paul2004mean} and \cite{li2005note}). Because of these applications, a number of researchers have studied stochastic comparisons of the second-order statistics in various stochastic senses. \cite{pualtuanea2008comparison} obtained hazard rate ordering between the second-order statistics from independent exponential random variables. \cite{ zhao2009likelihood} generalized the result of \cite{pualtuanea2008comparison} in the likelihood ratio ordering. \cite{zhao2009characterization} extended the
work of \cite{ zhao2009likelihood} in terms of the mean residual life order.
\cite{zhao2011right} obtained right spread ordering between the second-order statistics from heterogeneous exponential random variables. They also extended their results for the case of the proportional
decreasing hazard rate models. Similar results have been obtained in terms of the dispersive ordering by \cite{zhao2011dispersive}. Necessary and sufficient conditions for comparing the second-order statistics arising from exponential distributions have been derived by \cite{balakrishnan2015improved} in the sense of the mean residual life, dispersive, hazard rate, and likelihood ratio orderings. \cite{torrado2015tail} studied stochastic comparisons
between consecutive 2-within-m-out-of-n systems with components satisfying weak tail conditions for the proportional hazard rate model. \cite{cai2017hazard} developed comparison of the hazard rate functions of the second-order statistics under the assumption that the order statistics arise from two sets of independent multiple-outlier proportional hazard rates models.


\subsection{ELS model, its usefulness and plan of the paper}

A nonnegative absolutely continuous random variable $X$ is said to have exponentiated location-scale distribution if its cumulative distribution function is given by
\begin{eqnarray}\label{eq1.1}
F_{X}(x)\equiv
F_{X}(x|\lambda,\theta,\alpha;F_b) =\left[F_{b}\left(\frac{x-\lambda}{\theta}\right)\right]^{\alpha},~x>\lambda>0,~\alpha,~\theta>0.
\end{eqnarray}
The function $F_b$ is the baseline distribution function and $\lambda,~\theta,~\alpha$ are respectively the location, scale, shape parameters. The inclusion of a location parameter in the model is very useful because of many reasons.
In reliability analysis, it is very common that the failure of a unit does not occur immediately (say, $t=0$) just after allowing it working. Failure generally occurs after a duration of time (say, $t=\lambda>0$). Further, in finance, the claim of an insurance policy holder is possible after a specified period. It varies insurance company to insurance company. In medical study and economy, we often get datasets, which are skewed in nature. For example, in finance, we usually have small gains and occasionally, have a few large losses. This type of datasets is negatively skewed. To capture the skewness, we need skewness parameter. The shape parameter in the model plays a vital role to capture the skewness contained in the dataset. Denote $X\sim
\mathbb{ELS}(\lambda,\theta,\alpha;F_b)$ when the distribution function is given by (\ref{eq1.1}). Here, we use $f_{X}$, $\bar{F}_{X}=1-F_{X}$ and $r_{X}=f_{X}/\bar{F}_{X}$  to denote respectively the probability density, survival and hazard rate functions of the model. These functions for the baseline distribution are denoted by $f_b$, $\bar{F}_b=1-F_b$ and $r_b=f_b/\bar{F}_b$. Note that, the exponentiated location-scale model becomes the location-scale model, scale model and location model, for $\alpha=1$, $(\alpha,\lambda)=(1,0)$ and $(\alpha,\theta)=(1,1)$, respectively. Thus, the model with distribution function (\ref{eq1.1}) can accommodate a lot of statistical distributions. As a result, it is a flexible family of distributions.

After going through the literature, we observe that no study has been developed on the ordering properties of the second-order statistics from two sets of heterogeneous exponentiated location-scale distributed random variables so far. Our purpose in this paper is to investigate the existence of stochastic orderings between the second-order statistics. The rest of the paper is laid out as follows. 

In the next section, we review the stochastic orders, majorization and some related orders. We also present few lemmas which are  useful in proving main results. Section $3$ has  two subsections. In Subsection $3.1$, we obtain the usual stochastic and hazard rate orders between the second-order statistics under the assumption that the observations are independent. Subsection $3.2$ deals with the dependent observations coupled by Archimedean copulas. Here, we establish similar results. Applications of the established results in reliability and auction theory are presented in Section $4$. Finally, in Section $5$, some concluding remarks are reported.

Throughout the article, the random variables are assumed to be nonnegative and absolutely continuous. The derivatives always exist whenever they are used. Prime ($'$) stands for the derivative of a function, for example, $g'(x)=\frac{dg(x)}{dx}.$ Two terms increasing and decreasing are used in nonstrict sense. We use the notations:  $\mathcal{D}_{+}=\{(x_1,\cdots,x_n):x_{1}\geq
\cdots\geq x_{n}>0\}$,
$\mathcal{I}_{+}=\{(x_1,\cdots,x_n):0<x_{1}\leq \cdots\leq x_{n}\}$ and
$\bm{1}_{n}=(1,\cdots,1).$ Bold symbols will be used to denote vectors.

\section{Preliminaries\setcounter{equation}{0}}
In this section, we recall notions of stochastic orderings, majorization-based orderings, some definitions and lemmas, which are useful in subsequent sections. First, we present the concept of the usual stochastic and hazard rate orderings.
\subsection{Stochastic orders}
Consider two nonnegative absolutely continuous random variables $Z$ and $W$. Let the density functions, distribution functions, survival functions and hazard rate functions of $Z$ and $W$  be denoted by $f_{Z}$ and $f_{W}$, $F_{Z}$ and $F_{W}$, $\bar
F_{Z}=1-F_{Z}$ and $\bar F_{W}=1-F_{W}$,  $r_{Z}=f_{Z}/\bar
F_{Z}$ and $r_{W}=f_{W}/
\bar F_{W}$, respectively.
\begin{definition}
	A random variable $W$ is said to be smaller than $Z$ in the sense of the
	\begin{itemize}
		\item hazard rate order (written $W\leq_{hr}Z$),
		if $r_{Z}(x)\leq r_{W}(x)$, for all $x$.
		
		\item usual stochastic order (written $W\leq_{st}Z$), if
		$\bar F_{W}(x)\leq\bar F_{Z}(x)$, for all $x$.
	\end{itemize}
\end{definition}
The implication that the hazard rate order implies the usual stochastic order is obvious. For an excellent exposition on this topic, we refer the reader to \cite{shaked2007stochastic}.

\subsection{Majorization and related orders}
To establish various inequalities in different  areas of statistics, one of the easiest tools is the notion of majorization. In the following, we present definitions of majorization orders. Let $\boldsymbol{x} =
\left(x_{1},\cdots,x_{n}\right)$ and $\boldsymbol{y} =
\left(y_{1},\cdots,y_{n}\right)$ be two $n$ dimensional vectors taken from  $\mathbb{A},$ where $\mathbb{A} \subset \mathbb{R}^{n}$ and $\mathbb{R}^{n}$ be
an $n$-dimensional Euclidean space. The order coordinates of the
vectors $\boldsymbol{x}$ and $\boldsymbol{y}$ are denoted by  $x_{1:n}\leq \cdots \leq x_{n:n}$ and
$y_{1:n}\leq\cdots \leq y_{n:n},$  respectively.
\begin{definition}\label{def2.2}
	A vector $\boldsymbol{x}$ is said to be
	\begin{itemize}
		\item  majorized by the vector $\boldsymbol{y},$ (written
		 $\boldsymbol{x}\preceq^{m} \boldsymbol{y}$), if for each $l=1,\cdots,n-1$, we have
		$\sum_{i=1}^{l}x_{i:n}\geq \sum_{i=1}^{l}y_{i:n}$ and
		$\sum_{i=1}^{n}x_{i:n}=\sum_{i=1}^{n}y_{i:n}$.
		\item weakly submajorized by the vector $\boldsymbol{y},$ (written  $\boldsymbol{x}\preceq_{w} \boldsymbol{y}$), if for each $l=1,\cdots,n$, we have
		$\sum_{i=l}^{n}x_{i:n}\leq \sum_{i=l}^{n}y_{i:n}.$
		\item weakly supermajorized by  the vector $\boldsymbol{y},$ (written
		 $\boldsymbol{x}\preceq^{w} \boldsymbol{y}$), if for each $l=1,\cdots,n$, we have
		$\sum_{i=1}^{l}x_{i:n}\geq \sum_{i=1}^{l}y_{i:n}.$
		\item reciprocally majorized by the vector $\boldsymbol{y},$ (written
		 $\boldsymbol{x}\preceq^{rm}\boldsymbol{y}$), if $\sum_{i=1}^{l}x_{i:n}^{-1}\le
		\sum_{i=1}^{l}y_{i:n}^{-1}$, for all $l=1,\cdots,n$.
	\end{itemize}
\end{definition}
It is easy to verify that  $\boldsymbol{x}\preceq^{w} \boldsymbol{y}\Rightarrow \boldsymbol{x}\preceq^{rm} \boldsymbol{y}$. Also, majorization order implies both weakly supermajorization and weakly submajorization orders. For the interested readers, we refer to \cite{Marshall2011} for an extensive and comprehensive details on the theory of majorization and its applications in the field of statistics. Next, we provide the definition of the Schur-convex and Schur-concave functions. Note that the concept of majorization is concerned with Schur functions.

\begin{definition}\label{def2.3}
	A real-valued function $h:\mathbb{R}^n\rightarrow \mathbb{R}$ is said
	to be Schur-convex (Schur-concave) on $\mathbb{R}^n$, if
	$\boldsymbol {x}\overset{m}{\succeq}\boldsymbol{ y}\Rightarrow
	h(\boldsymbol { x})\geq( \leq)~h(\boldsymbol { y}), \text{ for all } \boldsymbol { x}, ~\boldsymbol
	{ y} \in \mathbb{R}^n.$
\end{definition}

Now, we present the following lemmas, which play a crucial role to assist the proof of the results in the subsequent sections.
{\begin{lemma}(\cite{kundu2016some})\label{lem2.1}
		Let $h:\mathcal{D}_+\rightarrow \mathbb{ R}$ be a function, continuously differentiable on the interior of $\mathcal{D}_+.$ Then, for $\boldsymbol{x},~\boldsymbol{y}\in\mathcal{D}_+,$
		$\boldsymbol{x}\succeq^{m}\boldsymbol{y}\text{ implies } h(\boldsymbol{x})\geq(\text{respectively},~\leq )~h(\boldsymbol{y})$
		if and only if,
		$h_{(k)}(\boldsymbol{ z}) \text{ is decreasing (respectively, increasing) in } k=1,\cdots,n,$
		where $h_{(k)}(\boldsymbol{ z})=\partial h(\boldsymbol{z})/\partial z_k$ denotes the partial derivative of $h$ with respect to its $k$th argument.
	\end{lemma}
\begin{lemma}(\cite{kundu2016some})\label{lem2.2}
	Let $h:\mathcal{I}_+\rightarrow \mathbb{ R}$ be a function, continuously differentiable on the interior of $\mathcal{I}_+.$ Then, for $\boldsymbol{x},~\boldsymbol{y}\in\mathcal{I}_+,$
	$\boldsymbol{x}\succeq^{m}\boldsymbol{y}\text{ implies } h(\boldsymbol{x})\geq(\text{respectively}, ~\leq ) ~h(\boldsymbol{y})$
	if and only if,
	$h_{(k)}(\boldsymbol{ z}) \text{ is increasing (respectively, decreasing) in } k=1,\cdots,n,$
	where $h_{(k)}(\boldsymbol{ z})=\partial h(\boldsymbol{z})/\partial z_k$ denotes the partial derivative of $h$ with respect to its $k$th argument.
\end{lemma}
\begin{lemma}(\cite{hazra2017stochastic})\label{lem2.3}
	Let $S\subseteq\mathbb{ R}^n_{+}.$
	Further, let $h:S\rightarrow \mathbb{ R}$ be a real-valued function. Then, for $\boldsymbol{x},~\boldsymbol{y}\in S,$
	$\boldsymbol{x}\succeq^{rm}\boldsymbol{y}\text{ implies } h(\boldsymbol{x})\geq(\text{respectively},~\leq )~h(\boldsymbol{y})$,
	if and only if,
	\begin{itemize}
		\item[(i)] $ h(\frac{1}{a_1},\cdots,\frac{1}{a_n})$ is Schur-convex (respectively, Schur-concave) in $(a_1,\cdots,a_n)\in S,$
		\item[(ii)] $ h(\frac{1}{a_1},\cdots,\frac{1}{a_n})$ is increasing (respectively, decreasing) in $a_i$,
		where $a_i=\frac{1}{x_i}$, for $i=1,\cdots,n.$
	\end{itemize}
\end{lemma}
\begin{lemma}(Lemma $3,$ \cite{balakrishnan2015stochastic})\label{lem2.4}
	Let the function $\omega:(0,\infty)\times(0,1)\rightarrow(0,\infty)$ be defined as
	$$\omega(\alpha,t)=\frac{\alpha(1-t)t^{\alpha}}{1-t^{\alpha}}.$$
	Then,
	\begin{itemize}
\item[(i)] for each $0<x<1,$ $\omega(\alpha,t)$ is decreasing in $\alpha;$
\item[(i)] for each $0<\alpha\leq1,$ $\omega(\alpha,t)$ is decreasing in $t;$
\item[(i)] for each $\alpha\geq 1,$ $\omega(\alpha,t)$ is increasing in $t.$
	\end{itemize}
\end{lemma}
 Consider a random vector $\boldsymbol{X}=(X_1,\cdots,X_n)$ with joint cumulative distribution function $F$ and joint survival function $\bar{F}$. Denote $\boldsymbol{v}=(v_1,\cdots,v_n)$. Suppose the functions $C(\boldsymbol{v}):[0,1]^n\rightarrow [0,1]$ and
$\hat {C}(\boldsymbol{v}):[0,1]^n\rightarrow [0,1]$ are such that for all $ x_i,~i\in \mathcal I_n, $ an index set,
$$ F(x_1,\cdots,x_n)=C(F_1(x_1),\cdots,F_n(x_n))~\mbox{and}~\bar{F}(x_1,\cdots,x_n)=\hat{C}(\bar{F_1}(x_1),\cdots,\bar{F_n}(x_n))$$ hold.
These functions $C(\boldsymbol{v})$ and $\hat{C}(\boldsymbol{v})$ are known as the  copula and survival copula of $\boldsymbol{X}$, respectively. Note that $F_1,\cdots,F_n$ and $\bar{F_1},\cdots,\bar{F_n}$ are the univariate marginal distribution and survival functions of $X_1,\cdots,X_n$, respectively.
 Consider a function $\psi:[0,\infty)\rightarrow[0,1]$, which is nonincreasing and continuous with $\psi(0)=1$ and $\psi(\infty)=0.$ If $\psi$ satisfies the conditions $(-1)^i{\psi}^{i}(x)\geq 0,~ i=0,1,\cdots,d-2$ and  $(-1)^{d-2}{\psi}^{d-2}$ is  convex and nonincreasing, then the generator $\psi$ is called $d$-monotone. Let us define $\psi={\phi}^{-1}$,  the right continuous inverse of $\psi$. Then, $C_{\psi}$ with generator $\psi$ is called the Archimedean copula if
 \begin{eqnarray}
 C_{\psi}(v_1,\cdots,v_n)=\psi\left[{\phi(v_1)},\cdots,\phi(v_n)\right],\text{ for all } v_i\in(0,1),~i=1,\cdots,n.
 \end{eqnarray}
 Interested readers may refer to \cite{nelsen2006introduction} and \cite{mcneil2009multivariate} for more details on Archimedean copulas.

\section{Main results\setcounter{equation}{0}}
This section deals with various ordering results between the second-order statistics. The order statistics arise from two independent groups of heterogeneous exponentiated location-scale distributed random variables. The random variables are assumed to be either statistically independent or interdependent. First, we consider the case of independent random samples.
\subsection{Independent case}
Let $\bm{X}=(X_{1},\cdots,X_{n})$ and $\bm{Y}=(Y_{1},\cdots,Y_{n})$ be two $n$-dimensional vectors of heterogeneous independent random observations, with $\bm{X}\sim \mathbb{ELS}(
\bm{\lambda},\bm{\theta},\bm{\alpha};F_{b})$ and $\bm{Y}\sim \mathbb{ELS}(
\bm{\mu},\bm{\delta},\bm{\beta};F_{b})$, where $\bm{\lambda}=(\lambda_1,\cdots,\lambda_n)$, $\bm{\theta}=
(\theta_1,\cdots,\theta_n)$, $\bm{\alpha}=(\alpha_1,\cdots,\alpha_n)$, $\bm{\mu}=(\mu_1,\cdots,\mu_n)$,
$\bm{\delta}=(\delta_1,\cdots,\delta_n)$ and $\bm{\beta}=(\beta_1,\cdots,\beta_n)$. Here, $X_{i}\sim \mathbb{ELS}(
\lambda_i,\theta_i,\alpha_i;F_{b})$ and $Y_i\sim \mathbb{ELS}(
\mu_i,\delta_i,\beta_i;F_{b})$, for $i=1,\cdots,n$.
We recall that $F_b$ is the cumulative distribution function of the baseline distribution of the exponentiated location-scale family of distributions. Under this general set-up, the survival functions of $X_{2:n}$ and $Y_{2:n}$ are respectively obtained as
\begin{eqnarray}\label{eq3.1}
\bar{F}_{X_{2:n}}(x)=\sum_{l=1}^{n}\left[\prod_{k\neq l}^{n}\left\{1-\left[F_{b}\left(\frac{x-\lambda_k}{\theta_k}\right)\right]^{\alpha_k}\right\}\right]-(n-1)\prod_{k=1}^{n}\left\{1-\left[F_{b}\left(\frac{x-\lambda_k}{\theta_k}\right)\right]^{\alpha_k}\right\},
\end{eqnarray}
where $x>\max(\lambda_1,\cdots,\lambda_n)$ and
\begin{eqnarray}
\small\bar{G}_{Y_{2:n}}(x)=\sum_{l=1}^{n}\left[\prod_{k\neq l}^{n}\left\{1-\left[F_{b}\left(\frac{x-\mu_k}{\delta_k}\right)\right]^{\beta_k}\right\}\right]-(n-1)\prod_{k=1}^{n}\left\{1-\left[F_{b}\left(\frac{x-\mu_k}{\delta_k}\right)\right]^{\beta_k}\right\},
\end{eqnarray}
where  $x>\max(\mu_1,\cdots,\mu_n).$ The following theorem shows that under some sufficient conditions, the fail-safe system associated with weakly submajorized location parameter vector produces less reliable system. It is worth to mention that the proofs of the results presented here are mainly based on the majorization and Schur functions.
\begin{theorem}\label{th3.1}
	Suppose $\bm{X}\sim \mathbb{ELS}(
	\bm{\lambda},\bm{\theta},\bm{\alpha};F_{b})$ and $\bm{Y}\sim \mathbb{ELS}(
	\bm{\mu},\bm{\delta},\bm{\beta};F_{b})$, with $\bm{\theta}=\bm{\delta}$ and $\bm{\alpha}=\bm{\beta}=\alpha \bm{1}_{n}\le 1.$ Let $\bm{\lambda},~\bm{\theta},~\bm{\mu}\in \mathcal{I}_{+}~(or~\mathcal{D}_{+})$ and $w r_{b}(w)$ be decreasing in $w>0.$ Then, ${\boldsymbol\lambda}\succeq_{w}{\boldsymbol\mu}\Rightarrow X_{2:n}\geq_{st}Y_{2:n}$.
\end{theorem}
\begin{proof}
	We present the proof when $\bm{\lambda},~\bm{\theta},~\bm{\mu}\in \mathcal{I}_{+}$. The proof is similar when these vectors belong to $\mathcal{D}_{+}.$ Denote $\zeta_1\left({\boldsymbol \lambda}\right)=\bar{F}_{X_{2:n}}(x)$, where under the present set-up, the survival function of $X_{2:n}$ can be obtained from (\ref{eq3.1}) after substituting $\alpha$ in place of $\alpha_k.$ The partial derivative of
	$\zeta_1\left({\boldsymbol \lambda}\right)$ with respect to $\lambda_i$, for $i=1,\cdots,n$ is
	\begin{equation}\label{eqi14}
	\frac{\partial\zeta_1\left({\boldsymbol \lambda}\right)}{\partial \lambda_i}=\vartheta_{1i}(\bm{\lambda})\vartheta_1(\lambda_i),
	\end{equation}
	where
	\begin{eqnarray}\label{eqi3.1}\nonumber
	\vartheta_{1i}(\bm{\lambda})&=&\sum_{\overset{l=1}{l\neq i}}^{n}\left[\prod_{k\neq l}^{n}\left\{1-\left[F_{b}\left(\frac{x-\lambda_k}{\theta_k}\right)\right]^{\alpha}\right\}\right]-(n-1)\prod_{k=1}^{n}\left\{1-\left[F_{b}\left(\frac{x-\lambda_k}{\theta_k}\right)\right]^{\alpha}\right\} \text{ and }
	\\
	\vartheta_1( \lambda_i)&=&\left[\frac{[w{{r}_{b}}(w)]_{w=\left(\frac{x-\lambda_i}{\theta_i}\right)}}{x-\lambda_i}\right]\left[\frac{\alpha F_{b}^{(\alpha-1)}\left(\frac{x-\lambda_i}{\theta_i}\right)\left[1-F_{b}\left(\frac{x-\lambda_i}{\theta_i}\right)\right]}{1-\left[F_{b}\left(\frac{x-\lambda_i}{\theta_i}\right)\right]^{\alpha}}\right].
	\end{eqnarray}
	To prove the required result, from Theorem $A.8$ of \cite{Marshall2011}, it is enough to show that  $\zeta_1(\boldsymbol{\lambda})$ is increasing and  Schur-convex with respect to $\boldsymbol{\lambda}\in \mathcal{I}_{+}.$ Let $1\le i \le j \le n.$ Thus, $\theta_i \le \theta_j$ and $\lambda_i\le \lambda_j.$ As a result, we have
	\begin{eqnarray}
	\frac{x-\lambda_i}{\theta_i}\ge \frac{x-\lambda_j}{\theta_j}~\mbox{and}~\left[F_{b}\left(\frac{x-\lambda_i}{\theta_i}\right)\right]^{\alpha}\ge\left [F_{b}\left(\frac{x-\lambda_j}{\theta_j}\right)\right]^{\alpha}.
   \end{eqnarray}
	Note that ``$\zeta_1(\boldsymbol{\lambda})$ is increasing and Schur-convex with respect to $\lambda_i$" is equivalent to show that both $\vartheta_{1i}(\bm{\lambda})$ and $\vartheta_1( \lambda_i)$ are positive-valued and increasing with respect to $\lambda_i,$ for $i=1,\cdots,n.$ Clearly, $\vartheta_1( \lambda_i)\ge 0.$ Further, $w r_{b}(w)$ is decreasing in $w>0.$ Thus,
\begin{equation}\label{eqi5}
[w r_{b}(w)]_{w=\left(\frac{x-\lambda_i}{\theta_i}\right)}\leq[w r_{b}(w)]_{w=\left(\frac{x-\lambda_j}{\theta_j}\right)}.
\end{equation}
Using Lemma \ref{lem2.4} and (\ref{eqi5}), it can be shown that $\vartheta_1( \lambda_i)$ is increasing with respect to $\lambda_i$, for $i=1,\cdots,n.$ Again,
$\vartheta_{1i}(\bm{\lambda})\ge 0,$ since
\begin{eqnarray}
\prod_{k\neq l}^{n}\left\{1-\left[F_{b}\left(\frac{x-\lambda_k}{\theta_k}\right)\right]^{\alpha}\right\}\geq \prod_{k= 1}^{n}\left\{1-\left[F_{b}\left(\frac{x-\lambda_k}{\theta_k}\right)\right]^{\alpha}\right\}.	
\end{eqnarray}	
Moreover, for $\lambda_i\le \lambda_j$, we have
\begin{align}
\vartheta_{1i}(\bm{\lambda})-\vartheta_{1j}(\bm{\lambda})=&\sum_{\overset{l=1}{l\neq i}}^{n}\left[\prod_{k\neq l}^{n}\left\{1-\left[F_{b}\left(\frac{x-\lambda_k}{\theta_k}\right)\right]^{\alpha}\right\}\right]-\sum_{\overset{l=1}{l\neq j}}^{n}\left[\prod_{k\neq l}^{n}\left\{1-\left[F_{b}\left(\frac{x-\lambda_k}{\theta_k}\right)\right]^{\alpha}\right\}\right]\nonumber\\
=&\prod_{k\neq \{i,j\}}^{n}\left\{1-\left[F_{b}\left(\frac{x-\lambda_k}{\theta_k}\right)\right]^{\alpha}\right\}\left[-\left[F_{b}\left(\frac{x-\lambda_i}{\theta_i}\right)\right]^{\alpha}+\left[F_{b}\left(\frac{x-\lambda_j}{\theta_j}\right)\right]^{\alpha}\right]\nonumber\\
\leq&~ 0.
\end{align}	
	Therefore, $\vartheta_{1i}(\bm{\lambda})$ is increasing with respect to $\lambda_i$. Further, from the definition of the Schur-convex function and Lemma \ref{lem2.2}, clearly,
	 $\zeta_{1}(\boldsymbol{\lambda})$ is Schur-convex with respect to $\boldsymbol{\lambda}\in \mathcal{I}_+.$ This completes the proof.
\end{proof}
It is known that
\begin{eqnarray}\label{eq3.9*}
(\mu_{1},\cdots,\mu_{n})\preceq_{w} \left(\frac{1+\mu_{1}}{2},\cdots,\frac{1+\mu_{n}}{2}\right)\preceq_{w}
(\mu_{m},\cdots,\mu_{m}),
\end{eqnarray}
where $\mu_{m}=\max\{\frac{1+\mu_{1}}{2},\cdots,\frac{1+\mu_{n}}{2}\}$. Thus, the following corollary is immediate from Theorem \ref{th3.1}. This is useful to obtain an upper bound for the survival function of a fail-safe system with heterogeneous components in terms of the lifetime of another fail-safe system with homogeneous components.
\begin{corollary}\label{cor3.1}
Let $\bm{X}\sim \mathbb{ELS}(
	\mu_{m},\theta,\alpha;F_{b})$ and $\bm{Y}\sim \mathbb{ELS}(
	\bm{\mu},\theta,\alpha;F_{b}),$ with $\bm{\mu}\in \mathcal{I}_{+}$ and $w r_{b}(w)$ be decreasing in $w>0.$ Then, for $\mu_{n}\le1,$
\begin{eqnarray}
\bar{F}_{Y_{2:n}}(x)\le \sum_{l=1}^{n}\left[\prod_{k\neq l}^{n}\left\{1-\left[F_{b}\left(\frac{x-\mu_{m}}{\theta}\right)\right]
^{\alpha}\right\}\right]-(n-1)\prod_{k=1}^{n}\left\{1-\left[F_{b}
\left(\frac{x-\mu_m}{\theta}\right)\right]^{\alpha}\right\}.
\end{eqnarray}
\end{corollary}

Below, we have another result, which readily follows from Theorem \ref{th3.1} and the fact that $(\mu,\cdots,\mu)
\preceq_{w}(\lambda_1,\cdots,\lambda_n)$ holds, for $n\mu \le \sum_{i=1}^{n}\lambda_i$,
\begin{corollary}
	In addition to the set-up as in Theorem \ref{th3.1}, we assume that $\mu_1=\cdots=\mu_n=\mu$. Then, $n\mu \le \sum_{i=1}^{n}\lambda_i\Rightarrow X_{2:n}\geq_{st}Y_{2:n}$.
\end{corollary}

To illustrate Theorem \ref{th3.1}, we consider the following example.
\begin{example}\label{ex3.1}
	Let $\{X_1,X_2,X_3\}$ and $\{Y_1,Y_2,Y_3\}$ be two sets of three independent random variables, with $X_{i}\sim F^{\alpha}_{b}(\frac{x-\lambda_i}{\theta_i})$ and $Y_{i}\sim F^{\alpha}_{b}(\frac{x-\mu_i}{\theta_i})$, for $i=1,2,3.$ Take the baseline distribution function as $F_{b}(x)=1-x^{-2},~x\ge1.$ Clearly, $xr_{b}(x)=2$, and hence $xr_{b}(x)$ is decreasing for $x\ge1$. Further, set $\bm{\theta}=\bm{\delta}=(0.5,0.7,0.9)$, $\bm{\lambda}=(5,7,9)$, $\bm{\mu}=(2,4,7)$ and $\alpha=0.2.$ It is easy to check that all other conditions of Theorem \ref{th3.1} are satisfied. Thus, we have $X_{2:3}\ge_{st}Y_{2:3}$, which  can be verified from Figure $1a$.
	\begin{figure}[h]
		\begin{center}
			\subfigure[]{\label{c1}\includegraphics[height=2.4in]{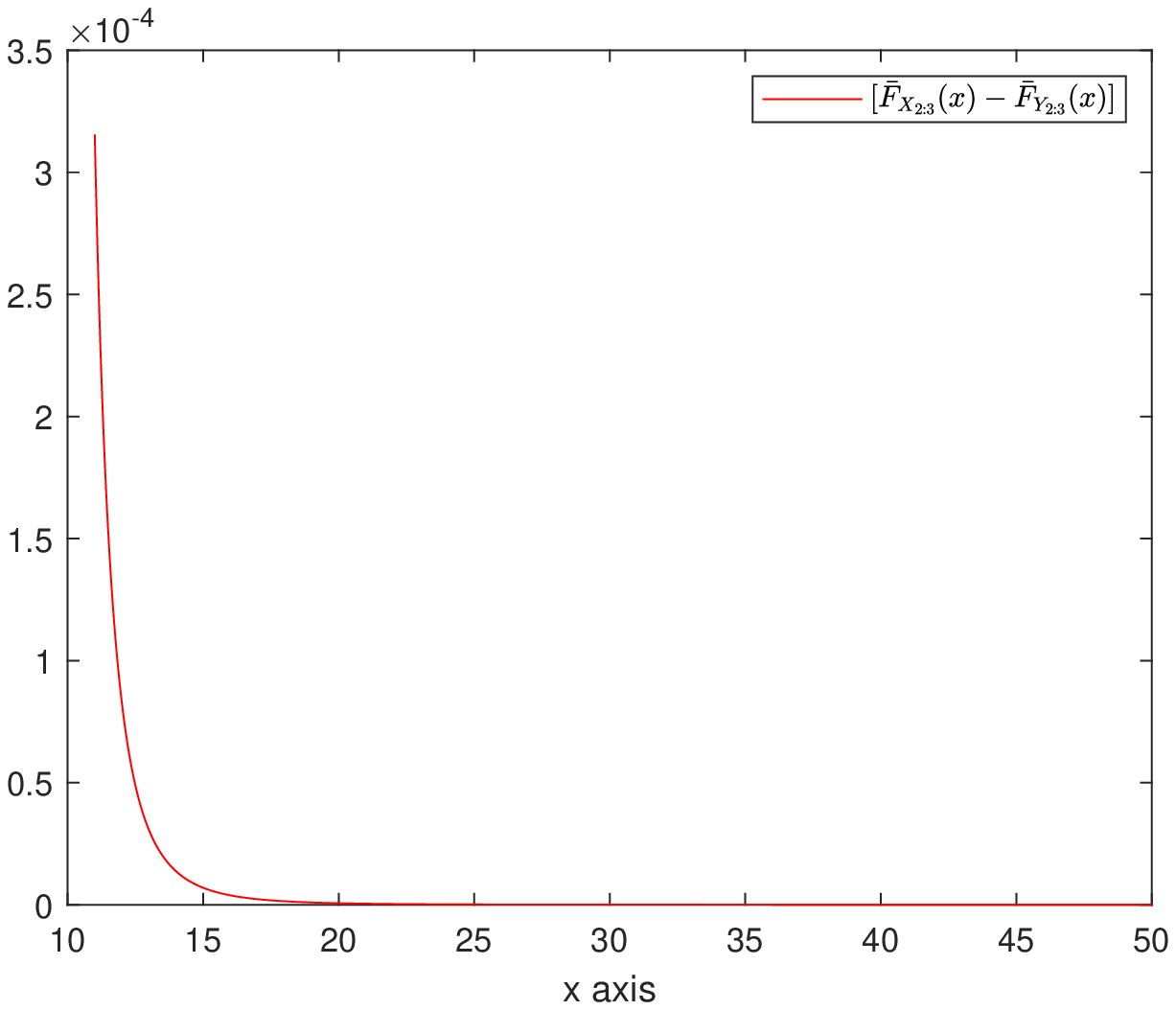}}
			\subfigure[]{\label{c1}\includegraphics[height=2.4in]{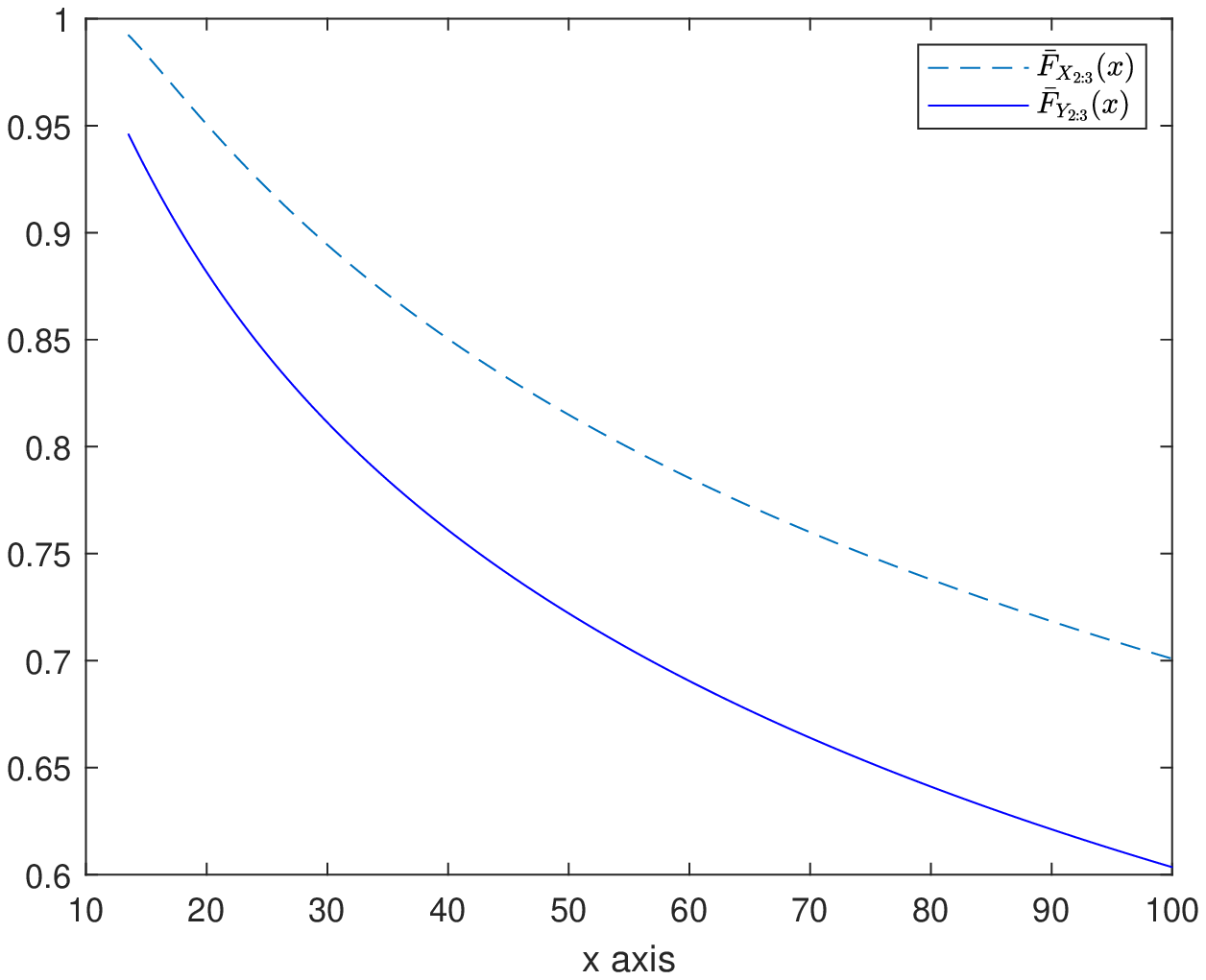}}
			\caption{(a) Plot of the difference $\bar{F}_{X_{2:3}}(x)-\bar{F}_{Y_{2:3}}(x)$ based on the observations as in Example \ref{ex3.1}. (b) Graphs of the survival functions, which are described in Example \ref{ex3.5}, for $a=0.12$. }
		\end{center}
	\end{figure}
\end{example}

In the next result, we consider the models with common scale and common shape parameters. Note that the previous theorem deals with the condition that $\boldsymbol{\theta}=\boldsymbol{\delta}$. Here, we take $\boldsymbol{\theta}=\boldsymbol{\delta}=\theta \bm{1}_{n}$. Under this restriction on $\boldsymbol{\theta}$ and $\boldsymbol{\delta}$, we have the following dominance result between the second-order statistics, similar to Theorem \ref{th3.1}. The proof is analogous to that of Theorem \ref{th3.1} and thus, it is omitted for the sake of brevity.
\begin{theorem}\label{th3.1*}
 Assume $\boldsymbol{X}\sim \mathbb{ELS}(
\boldsymbol{\lambda},\boldsymbol{\theta},\boldsymbol{\alpha};F_{b})$ and $\boldsymbol{Y}\sim \mathbb{ELS}(
\boldsymbol{\mu},\boldsymbol{\delta},\boldsymbol{\beta};F_{b}),$ with $\boldsymbol{\theta}=\boldsymbol{\delta}=\theta \bm{1}_{n}$ and $\boldsymbol{\alpha}=\boldsymbol{\beta}=\alpha \bm{1}_{n}\leq1$. Let $\boldsymbol{\lambda},~ \boldsymbol{\mu}\in\mathcal{I}_+~(or~\mathcal{D}_+)$ and $r_{b}(w)$ be decreasing in $w>0$. Then,  ${\boldsymbol\lambda}\succeq_{w}{\boldsymbol\mu}\Rightarrow X_{2:n}\geq_{st}Y_{2:n}.$
\end{theorem}

The below presented numerical example is an application of Theorem \ref{th3.1*}. 
\begin{example}
	Let us take two sets of independent observations  $\{X_{1},X_{2},X_{3}\}$ and
	$\{Y_{1},Y_{2},Y_{3}\}$ such that $X_{i}\sim	 F^{\alpha}_b(\frac{x-\lambda_i}{\theta})$ and $Y_{i}\sim
	F^{\alpha}_b(\frac{x-\mu_i}{\theta})$, for $i=1,2,3.$ The baseline distribution  is taken as the Burr distribution, where
	$F_b(x)=1-(1+x^c)^{-k},~x>0,~c>0,~k>0.$ For $0< c\leq 1$ and $ k>0$, the hazard rate function of the Burr distribution is decreasing.
	Now, consider $\boldsymbol{\lambda}=(7,9,11)$,
	$\boldsymbol{\mu}=(3,5,8)$, $\alpha=0.2,$ $c=0.5,~ k=2$ and
	$\theta=2$. Clearly, $\boldsymbol{\lambda}\succeq_{w}\boldsymbol{\mu}$ and $\boldsymbol{\lambda}, ~ \boldsymbol{\mu}\in\mathcal{I}_{+}$.
	Therefore, utilizing Theorem \ref{th3.1*},
	we obtain $X_{2:3}\ge_{st}Y_{2:3}.$
\end{example}

In Theorems \ref{th3.1} and \ref{th3.1*}, we assume different location parameter vectors. In the following result, the scale parameter vectors of the models are taken to be different. We obtain weak supermajorized-based sufficient conditions, under which the second-order statistics are compared in terms of the usual stochastic order.
\begin{theorem}\label{th3.2}
	Suppose $\bm{X}\sim \mathbb{ELS}(
	\bm{\lambda},\bm{\theta},\bm{\alpha};F_{b})$ and $\bm{Y}\sim \mathbb{ELS}(
	\bm{\mu},\bm{\delta},\bm{\beta};F_{b}),$ with $\bm{\lambda}=\bm{\mu}$ and $\bm{\alpha}=\bm{\beta}=\alpha \bm{1}_{n}\le 1.$ Let $\bm{\lambda},~\bm{\theta},~\bm{\delta}\in \mathcal{I}_{+}~(or~\mathcal{D}_{+})$ and $w r_{b}(w)$ be decreasing in $w>0.$ Then, $\frac{1}{\boldsymbol\theta}\succeq^{w}\frac{1}{\boldsymbol\delta}\Rightarrow X_{2:n}\geq_{st}Y_{2:n}$.
\end{theorem}
\begin{proof}
	We present the proof for $\bm{\lambda},~\bm{\theta},~\bm{\delta}\in \mathcal{I}_{+}$. The proof for the other case is similar, and hence not presented here. Under the assumed set-up, the survival function of $X_{2:n}$ can be expressed as
	\begin{eqnarray}\label{eq3.11}
	\small\bar{F}_{X_{2:n}}(x)=\sum_{l=1}^{n}\left[\prod_{k\neq l}^{n}\left\{1-\left[F_{b}\left(({x-\lambda_k}){v_k}\right)\right]^{\alpha}\right\}\right]-(n-1)\prod_{k=1}^{n}\left\{1-\left[F_{b}\left(({x-\lambda_k}){v_k}\right)\right]^{\alpha}\right\},
	\end{eqnarray}
	where $x>\max(\lambda_1,\cdots,\lambda_n)$ and $v_i=1/\theta_i,$ for $i=1,\cdots,n.$ Define $\zeta_{2}(\bm{v})=\bar{F}_{X_{2:n}}(x),$ where $\bar{F}_{X_{2:n}}(x)$ is given by \eqref{eq3.11}. On differentiating this partially with respect to $v_i$, we get
	\begin{equation}\label{eqi15}
	\frac{\partial\zeta_2\left({\boldsymbol v}\right)}{\partial v_i}=-\vartheta_{2i}(\bm{v})\vartheta_2(v_i),
	\end{equation}
	where
	\begin{eqnarray}\nonumber\label{eqi3.9}
	\vartheta_{2i}(\bm{v})&=&\sum_{\overset{l=1}{l\neq i}}^{n}\left[\prod_{k\neq l}^{n}\left\{1-\left[F_{b}\left(({x-\lambda_k}){v_k}\right)\right]^{\alpha}\right\}\right]-(n-1)\prod_{k=1}^{n}\left\{1-\left[F_{b}\left(({x-\lambda_k}){v_k}\right)\right]^{\alpha}\right\} \text{ and }
	\\
	\vartheta_2( v_i )&=&\left[\frac{[w{{r}_{b}}(w)]_{w=\left(({x-\lambda_i}){v_i}\right)}}{v_i}\right]\left[\frac{\alpha F_{b}^{(\alpha-1)}\left(({x-\lambda_i}){v_i}\right)\left[1-F_{b}\left(({x-\lambda_i}){v_i}\right)\right]}{1-\left[F_{b}\left(({x-\lambda_i}){v_i}\right)\right]^{\alpha}}\right].	 
	\end{eqnarray}
	In order to prove the required result, it is sufficient to show that the function $\zeta_2\left({\boldsymbol v}\right)$ is decreasing and Schur-convex with respect to $\bm{v}\in \mathcal{D}_{+}.$ This can be accomplished with similar arguments, which have been used to prove Theorem \ref{th3.1}. Thus, the remaining part of the proof is omitted.
\end{proof}

We have $(\frac{1}{\theta_1},\cdots,\frac{1}{\theta_n})\succeq^{w} (\frac{1}{\theta},\cdots,\frac{1}{\theta})$, for $\frac{1}{\theta}\ge \frac{1}{n}\sum_{i=1}^{n}\frac{1}{\theta_i}$. Thus, we get the following corollary from Theorem \ref{th3.2}. This result is useful to obtain a lower bound of the reliability of a fail-safe system with heterogeneous components in terms of that of a fail-safe system with homogeneous components.
\begin{corollary}
	Let  $\boldsymbol{X}\sim \mathbb{ELS}(
	{\lambda},\boldsymbol{\theta},\boldsymbol{\alpha};{F_{b}})$ and $\boldsymbol{Y}\sim \mathbb{ELS}(
	\lambda,\theta,\boldsymbol{\beta};F_{b}),$ with  $\boldsymbol{\alpha}=\boldsymbol{\beta}=\alpha \bm{1}_n\le 1$. Also, consider that $\boldsymbol{\theta}\in\mathcal{I}_+~(or~\mathcal{D}_+)$. Then,
	$n/\theta\ge \sum_{i=1}^{n}(1/\theta_i)\Rightarrow X_{2:n}\geq_{st}Y_{2:n}$, provided $wr_{b}(w)$ is decreasing in $w>0$.
\end{corollary}

In the preceding theorem, we assume that the location parameters of the sets are vector-valued but equal. In the following result, we consider that the location parameters have a common (scalar) value. The proof is omitted, since it follows using a similar argument to that of Theorem \ref{th3.2}.
\begin{theorem}\label{th3.2*}
	Suppose  $\boldsymbol{X}\sim \mathbb{ELS}(
\boldsymbol{\lambda},\boldsymbol{\theta},\boldsymbol{\alpha};F_{b})$ and $\boldsymbol{Y}\sim \mathbb{ELS}(
\boldsymbol{\mu},\boldsymbol{\delta},\boldsymbol{\beta};F_{b}),$ with $\boldsymbol{\lambda}=\boldsymbol{\mu}=\lambda \bm{1}_{n}$ and $\boldsymbol{\alpha}=\boldsymbol{\beta}=\alpha \bm{1}_{n}\leq1$. Also, assume $ \boldsymbol{\theta},~\boldsymbol{\delta}\in\mathcal{I}_+~(or~\mathcal{D}_+)$ and $r_{b}(w)$ is decreasing in $w>0$. Then,  $1/{\boldsymbol\theta}\succeq^{w}{1/{\boldsymbol\delta}}\Rightarrow X_{2:n}\geq_{st}Y_{2:n}$.
\end{theorem}
The following constructed numerical example is an application of Theorem \ref{th3.2*}.
\begin{example}
	Let $\{X_{1},X_{2},X_{3}\}$ and
	$\{Y_{1},Y_{2},Y_{3}\}$ be two sets of independent random observations, with $X_{i}\sim	 F^{\alpha}_b(\frac{x-\lambda}{\theta_{i}})$ and $Y_{i}\sim
	F^{\alpha}_b(\frac{x-\lambda}{\delta_{i}})$, for $i=1,2,3.$ Take the baseline distribution function as the power generalized Weibull distribution, with
	$F_{b}(x)=1-\exp(1-(1+x^c)^{1/k}),~x>0,~c>0,~k>0.$ It is easy to check that for $ c\leq k$, $c<1,$ the hazard rate function of power generalized Weibull distribution is decreasing. Set $\alpha=0.6,$
	$\lambda=2$, $c=0.5,~k=2,$ $\boldsymbol{\theta}=(5,6,7)$ and $\boldsymbol{\delta}=(2,3,4)$. It can be verified that $1/{\boldsymbol\theta}\succeq^{w}{1/{\boldsymbol\delta}}$ and $\boldsymbol{\theta}, ~\boldsymbol{\delta}\in\mathcal{I}_{+}$.
	Therefore, from Theorem \ref{th3.2*},
	we have $X_{2:3}\ge_{st}Y_{2:3}.$
\end{example}

Next, we obtain another set of sufficient conditions, under which $X_{2:n}$ dominates $Y_{2:n}$ in the sense of the usual stochastic order. The conditions are associated with the reciprocal majorization order between the reciprocal of the scale parameter vectors.
\begin{theorem}\label{th3.3}
	Assume that  $\boldsymbol{X}\sim \mathbb{ELS}(
	\boldsymbol{\lambda},\boldsymbol{\theta},\boldsymbol{\alpha};F_{b})$ and $\boldsymbol{Y}\sim \mathbb{ELS}(
	\boldsymbol{\mu},\boldsymbol{\delta},\boldsymbol{\beta};F_{b}),$ with $\boldsymbol{\lambda}=\boldsymbol{\mu}$ and $\boldsymbol{\alpha}=\boldsymbol{\beta}=\alpha \bm{1}_{n}\leq1$. Also, let $\boldsymbol{\lambda}, ~\boldsymbol{\theta},~\boldsymbol{\delta}\in\mathcal{I}_+~(or~\mathcal{D}_+)$ and $w^2 r_{b}(w)$ be decreasing in $w>0$. Then, $1/{\boldsymbol\theta}\succeq^{rm}{1/{\boldsymbol\delta}}\Rightarrow X_{2:n}\geq_{st}Y_{2:n}$.
\end{theorem}
\begin{proof}
Similar to the above theorems, here, we also present the proof for $\boldsymbol{\lambda},~ \boldsymbol{\theta},~\boldsymbol{\delta}\in\mathcal{I}_{+}.$ The other case is similar, and hence omitted. Denote $\zeta_{3}(1/\bm{\theta})=\bar{F}_{X_{2:n}}(x).$ Differentiating this partially with respect to $\theta_{i}$, for $i=1,\cdots,n$, we obtain
\begin{equation}\label{eqi27}
\frac{\partial\zeta_2\left(\frac{1}{\boldsymbol \theta}\right)}{\partial \theta_i}=\vartheta_{3i}\left(\frac{1}{\bm{\theta}}\right)\vartheta_3\left(\frac{1}{\theta_i}\right),
\end{equation}
where
\begin{eqnarray}\nonumber\label{eqi3.6}
\vartheta_{3i}\left(\frac{1}{\bm{\theta}}\right)&=&\sum_{\overset{l=1}{l\neq i}}^{n}\left[\prod_{{k\neq l}
}^{n}\left\{1-\left[F_{b}\left(\frac{x-\lambda_k}{\theta_k}\right)\right]^{\alpha}\right\}\right]-(n-1)\prod_{k=1}^{n}\left\{1-\left[F_{b}\left(\frac{x-\lambda_k}{\theta_k}\right)\right]^{\alpha}\right\}\text{ and }
\\
\vartheta_3\left( \frac{1}{\theta_i}\right)&=&\left[\frac{[w^2{{r}_{b}}(w)]_{w=\left(\frac{x-\lambda_i}{\theta_i}\right)}}{x-\lambda_i}\right]\left[\frac{\alpha F_{b}^{(\alpha-1)}\left(\frac{x-\lambda_i}{\theta_i}\right)\left[1-F_{b}\left(\frac{x-\lambda_i}{\theta_i}\right)\right]}{1-\left[F_{b}\left(\frac{x-\lambda_i}{\theta_i}\right)\right]^{\alpha}}\right].
\end{eqnarray}
To prove the stated result, it is enough to show that $\zeta_2\left(\frac{1}{\boldsymbol \theta}\right)$ is increasing and Schur-convex with respect to $\bm{\theta}\in \mathcal{I}_{+}.$ This is similar to the proof of Theorem \ref{th3.1}. Thus, it is omitted.
\end{proof}
The next corollary is a consequence of Theorem \ref{th3.3} and the result that $(\frac{1}{\theta_1},\cdots,\frac{1}{\theta_n})\succeq^{rm}(\frac{1}{\delta},\cdots,\frac{1}{\delta})$ holds for $n\delta\le \sum_{i=1}^{n}\theta_i$.
\begin{corollary}
	Let $\boldsymbol{X}\sim \mathbb{ELS}(
	\boldsymbol{\lambda},\boldsymbol{\theta},\alpha;F_{b})$ and $\boldsymbol{Y}\sim \mathbb{ELS}(
	\boldsymbol{\mu},\delta,\alpha;F_{b}),$ with $\boldsymbol{\lambda}=\boldsymbol{\mu}$ and $\alpha\leq1$. Also, let $ \boldsymbol{\lambda},~\boldsymbol{\theta}\in\mathcal{I}_+~(or~\mathcal{D}_+)$ and $w^2 r_{b}(w)$ be decreasing in $w>0$. Then, $n\delta\le \sum_{i=1}^{n}\theta_i\Rightarrow X_{2:n}\geq_{st}Y_{2:n}$.
\end{corollary}
In the following result, we assume that both location and scale parameter vectors are different for different collections of random samples. The first part can be proved from Theorems \ref{th3.1} and \ref{th3.2}. The second one can be established based on Theorems \ref{th3.1} and \ref{th3.3}.
\begin{theorem}\label{th3.4}
	Suppose  $\boldsymbol{X}\sim \mathbb{ELS}(
	\boldsymbol{\lambda},\boldsymbol{\theta},\boldsymbol{\alpha};F_{b})$ and $\boldsymbol{Y}\sim \mathbb{ELS}(
	\boldsymbol{\mu},\boldsymbol{\delta},\boldsymbol{\beta};F_{b}),$ with $\boldsymbol{\alpha}=\boldsymbol{\beta}=\alpha \bm{1}_{n}\leq1$. Also, assume $\boldsymbol{\lambda}, ~\boldsymbol{\mu},~ \boldsymbol{\theta},~\boldsymbol{\delta}\in\mathcal{I}_+~(or~\mathcal{D}_+)$ and $w r_{b}(w)$ is decreasing in $w>0$. Then,
	\begin{itemize}
	 \item[(i)] ${\boldsymbol\lambda}\succeq_{w}{\boldsymbol\mu}$ and $1/{\boldsymbol\theta}\succeq^{w}{1/{\boldsymbol\delta}}\Rightarrow X_{2:n}\geq_{st}Y_{2:n}$.
	 \item[(ii)] ${\boldsymbol\lambda}\succeq_{w}{\boldsymbol\mu}$ and $1/{\boldsymbol\theta}\succeq^{rm}{1/{\boldsymbol\delta}}\Rightarrow X_{2:n}\geq_{st}Y_{2:n}$, provided $w^2 r_{b}(w)$ is decreasing in $w>0$.
	 \end{itemize}
\end{theorem}

Till now, we have proved various results to compare the second-order statistics in terms of the usual stochastic order. Next, we derive sufficient conditions for the comparison of the second-order statistics with respect to the hazard rate order. Note that the hazard rate order is of particular interest in the field of reliability theory because of the importance of the hazard rate function for various systems.  When the shape parameters are assumed to be $1$, the survival function of $X_{2:n}$ can be re-written as
\begin{align}\label{eq3.16}
	\bar{F}_{X_{2:n}}(x)&=
	\sum\limits_{i=1}^{n}\prod_{j\neq i}^{n}\bar{F_{b}}\left(\frac{x-\lambda_j}{\theta_j}\right)-(n-1)
\prod_{i=1}^{n}\bar{F_{b}}\left(\frac{x-\lambda_i}{\theta_i}\right)\nonumber\\
	 &=\prod_{i=1}^{n}\bar{F_{b}}\left(\frac{x-\lambda_i}{\theta_i}\right)\left[\sum\limits_{i=1}^{n}
\frac{1}{\bar{F_{b}}\left(\frac{x-\lambda_i}{\theta_i}\right)}-(n-1)\right]\nonumber\\
	 &=\prod_{i=1}^{n}\bar{F_{b}}\left(\frac{x-\lambda_i}{\theta_i}\right)\left[\sum\limits_{i=1}^{n}\frac{r_{b}
\left(\frac{x-\lambda_i}{\theta_i}\right)}{\tilde{r}_{b}\left(\frac{x-\lambda_i}{\theta_i}\right)}+1\right],
\end{align}
where $\tilde{r}_{b}$ is the reversed hazard rate function of the baseline distribution.  Thus, the hazard rate function of the second-order statistic $X_{2:n}$ can be obtained as
\begin{equation}\label{eq3.17}
r_{X_{2:n}}(x)=\sum\limits_{i=1}^{n}\frac{1}{\theta_{i}}r_{b}\left(\frac{x-\lambda_{i}}{\theta_{i}}\right)-
\left[\sum\limits_{i=1}^{n}\frac{1}{\theta_{i}}\left[\frac{r_b\left(\frac{x-\lambda_{i}}{\theta_{i}}\right)}
{\tilde{r}_b\left(\frac{x-\lambda_{i}}{\theta_{i}}\right)}\right]'\right]\left[\sum\limits_{i=1}^{n}\left[\frac{r_b
\left(\frac{x-\lambda_{i}}{\theta_{i}}\right)}{\tilde{r}_b\left(\frac{x-\lambda_{i}}{\theta_{i}}\right)}\right]
+1\right]^{-1}.
\end{equation}

The following result establishes that under some conditions, majorized location parameter vector produces system with smaller hazard rate. An interpretation of the next theorem is as follows. Let us consider two fail-safe systems $S$-I and $S$-II with components' lifetime vectors $\bm{X}$ and $\bm{Y}$, respectively. Under the condition that $S$-I and $S$-II have survived up to an arbitrary time point $t,$ the system with components' lifetime vector $\bm{X}$ (here, $S$-I) is more likely to fail in the immediate future than that with components' lifetime vector $\bm{Y}$ (here, $S$-II).
\begin{theorem}\label{th3.5}
	Let  $\boldsymbol{X}\sim \mathbb{ELS}(
	\boldsymbol{\lambda},\theta,\boldsymbol{\alpha};{F_{b}})$ and $\boldsymbol{Y}\sim \mathbb{ELS}(
	\boldsymbol{\mu},\theta,\boldsymbol{\beta};F_{b}),$ with  $\boldsymbol{\alpha}=\boldsymbol{\beta}=\bm{1}_n$. Also, consider $\boldsymbol{\lambda},~\boldsymbol{\mu}\in\mathcal{D}_+~(or~\mathcal{I}_+)$. Then,
	${\boldsymbol\lambda}\succeq^{m}{\boldsymbol\mu}\Rightarrow X_{2:n}\leq_{hr}Y_{2:n}$, provided $r_{b}(w)$ is decreasing, convex, $r_{b}(w)/\tilde{r}_{b}(w)$ is increasing, convex and $[r_{b}(w)/\tilde{r}_{b}(w)]''$ is decreasing in $w>0$.
\end{theorem}
\begin{proof}
Under the present set-up, the hazard rate function of $X_{2:n}$ is given by
\begin{equation}\label{eq3.18}
r_{X_{2:n}}(x)=\sum\limits_{i=1}^{n}\frac{1}{\theta}r_{b}\left(\frac{x-\lambda_{i}}
{\theta}\right)-\left[\sum\limits_{i=1}^{n}\frac{1}{\theta}g^{'}\left(\frac{x-\lambda_{i}}
	{\theta}\right)\right]\left[\sum\limits_{i=1}^{n}g\left(\frac{x-\lambda_{i}}{\theta}\right)+1\right]^{-1},
\end{equation}
where $g\left(\frac{x-\lambda_{i}}
{\theta}\right)=r_{b}\left(\frac{x-\lambda_{i}}
	{\theta}\right)/\tilde{r}_{b}\left(\frac{x-\lambda_{i}}
	{\theta}\right).$ On differentiating (\ref{eq3.18}) partially with respect to $\lambda_i$, for $i=1,\cdots,n$, we obtain
	\begin{align}\label{eq3.19}
	\frac{\partial r_{X_{2:n}}(x)}{\partial \lambda_{i}}=&-\frac{1}{{\theta}^2}r^{'}_{b}\left(\frac{x-\lambda_i}{\theta}\right)+\left[\frac{1}{{\theta}^2}g^{''}\left(\frac{x-\lambda_{i}}{\theta}\right)\right]\left[\sum\limits_{i=1}^{n}g\left(\frac{x-\lambda_{i}}{\theta}\right)+1\right]^{-1}\nonumber\\
	 ~&-\left[\frac{1}{\theta}g^{'}\left(\frac{x-\lambda_{i}}{\theta}\right)\sum\limits_{i=1}^{n}\frac{1}{\theta}g^{'}\left(\frac{x-\lambda_{i}}{\theta}\right)\right]\left[\sum\limits_{i=1}^{n}g\left(\frac{x-\lambda_{i}}{\theta}\right)+1\right]^{-2}.
	\end{align}
	The proof will be completed, if we can show that the function $r_{X_{2:n}}(x)$ given by (\ref{eq3.18}) is Schur-convex with respect to $\boldsymbol{\lambda}\in\mathcal{D}_+~(or~\mathcal{I}_+).$ Now, consider
	\begin{align}\label{eq3.20}
	\frac{\partial r_{X_{2:n}}(x)}{\partial \lambda_i}-\frac{\partial r_{X_{2:n}}(x)}{\partial \lambda_j}=&-\frac{1}{{\theta}^2}{r}_{b}'\left(\frac{x-\lambda_i}{\theta}\right)+\frac{1}{{\theta}^2}{r}_{b}'\left(\frac{x-\lambda_j}{\theta}\right)\nonumber\\
	 &+\left[\frac{1}{{\theta}^2}g^{''}\left(\frac{x-\lambda_{i}}{\theta}\right)-\frac{1}{{\theta}^2}g^{''}\left(\frac{x-\lambda_{j}}{\theta}\right)\right]\left[\sum\limits_{i=1}^{n}g\left(\frac{x-\lambda_{i}}{\theta}\right)+1\right]^{-1}\nonumber\\
	 &+\left[-\frac{1}{\theta}g^{'}\left(\frac{x-\lambda_{i}}{\theta}\right)+\frac{1}{\theta}g^{'}\left(\frac{x-\lambda_{j}}{\theta}\right)\right]\left[\sum\limits_{i=1}^{n}\frac{1}{\theta}g^{'}\left(\frac{x-\lambda_{i}}{\theta}\right)\right]\nonumber\\
	 ~&\times\left[\sum\limits_{i=1}^{n}g\left(\frac{x-\lambda_{i}}{\theta}\right)+1\right]^{-2}.
	\end{align}
	Let $\bm{\lambda}\in \mathcal{I}_{+}$.  Then, for $1\le i\le j \le n,$ we have $\lambda_i\le \lambda_j$ implies $\frac{x-\lambda_i}{\theta}\ge \frac{x-\lambda_j}{\theta}$. It is assumed that $r_{b}(w)$ is decreasing and convex. Thus, $-r_{b}'(w)|_{w=\frac{x-\lambda_i}{\theta}}\le -r_{b}'(w)|_{w=\frac{x-\lambda_j}{\theta}}$. Further, $r_{b}(w)/\tilde{r}_{b}(w)$ is increasing and convex. As a result, $-\frac{1}{\theta}g'(w)|_{w=\frac{x-\lambda_i}{\theta}}\le -\frac{1}{\theta}g'(w)|_{w=\frac{x-\lambda_j}{\theta}}$. Again, $[r_{b}(w)/\tilde{r}_{b}(w)]''$ is decreasing. So, $\frac{1}{\theta^2}g''(w)|_{w=\frac{x-\lambda_i}{\theta}}\le \frac{1}{\theta^2}g''(w)|_{w=\frac{x-\lambda_j}{\theta}}$. Combining these inequalities, we obtain $$	 \frac{\partial r_{X_{2:n}}(x)}{\partial \lambda_i}-\frac{\partial r_{X_{2:n}}(x)}{\partial \lambda_j}\leq0,$$  for $\bm{\lambda}\in \mathcal{I}_+.$ Thus, the rest of the proof readily follows from Lemma \ref{lem2.2}.
\end{proof}

By utilizing Lemma $2(i)$ of \cite{hazra2017stochastic}, we get the following corollary, which is a direct consequence of Theorem \ref{th3.5}. Particularly, here, we study comparisons of the lifetimes of fail-safe systems, one with heterogeneous components and other with identical components.
\begin{corollary}\label{col3.5}
	Consider two sets of independent random vectors $\boldsymbol{X}\sim \mathbb{ELS}(
	\boldsymbol{\lambda},\theta,\boldsymbol{\alpha};{F_{b}})$ and $\boldsymbol{Y}\sim \mathbb{ELS}(
	\lambda,\theta,\boldsymbol{\beta};F_{b}),$ with  $\boldsymbol{\alpha}=\boldsymbol{\beta}=\bm{1}_n$. Let $\boldsymbol{\lambda}\in\mathcal{D}_+~(or~\mathcal{I}_+)$. Then,
	$\lambda=\sum_{i=1}^{n}\lambda_i/n\Rightarrow X_{2:n}\leq_{hr}Y_{2:n}$, provided $r_{b}(w)$ is decreasing, convex, $r_{b}(w)/\tilde{r}_{b}(w)$ is increasing, convex and $[r_{b}(w)/\tilde{r}_{b}(w)]''$ is decreasing with respect to $w>0$.
\end{corollary}
This corollary can be used to obtain a lower bound of the hazard rate function of a fail-safe system having independent heterogeneous components in terms of the hazard rate function of a fail-safe system with independent identical components. A lower bound is given by
\begin{eqnarray}
r_{X_{2:n}}(x)\ge \frac{n}{\theta}r_{b}\left(\frac{x-\lambda}
{\theta}\right)-\left[\frac{n}{\theta}g^{'}\left(\frac{x-\lambda}
{\theta}\right)\right]\left[ng\left(\frac{x-\lambda}{\theta}\right)+1\right]^{-1},
\end{eqnarray}
where $g\left(\frac{x-\lambda}
{\theta}\right)=r_{b}\left(\frac{x-\lambda}
{\theta}\right)/\tilde{r}_{b}\left(\frac{x-\lambda}
{\theta}\right).$ It is noted that the conditions for the baseline distribution in Corollary \ref{col3.5} are satisfied by the Pareto distribution function $F(x)=1-x^{-a},~x\ge 1,~1<a<2.$ For this distribution, we have
\begin{eqnarray}
r_{X_{2:n}}(x)\geq \frac{{n}\left[a+(an-1)\left[\left(\frac{x-\lambda}{\theta}\right)^a+1\right]\right]}{\theta\left[n\left(\frac{x-\lambda}{\theta}\right)^a+1-n\right]} .
\end{eqnarray}

The following result provides conditions, under which the hazard rate ordering holds between the second-order statistics. Here, we have taken a common scale parameter vector for both sets of observations.
\begin{theorem}\label{th3.6}
	Let  $\boldsymbol{X}\sim \mathbb{ELS}(
	\boldsymbol{\lambda},\boldsymbol{\theta},\boldsymbol{\alpha};{F_{b}})$ and $\boldsymbol{Y}\sim \mathbb{ELS}(
	\boldsymbol{\mu},\boldsymbol{\delta},\boldsymbol{\beta};F_{b}),$ with $\boldsymbol{\delta}=\boldsymbol{\theta}$ and $\boldsymbol{\alpha}=\boldsymbol{\beta}=\bm{1}_n$. Also, consider $\boldsymbol{\lambda},~ \boldsymbol{\theta},~\boldsymbol{\mu}\in\mathcal{D}_+~(or~\mathcal{I}_+)$. Then,
	${\boldsymbol\lambda}\succeq^{m}{\boldsymbol\mu}\Rightarrow X_{2:n}\leq_{hr}Y_{2:n}$, provided $r_{b}(w)$ is decreasing, $w^2r'_{b}(w)$ is increasing,  $r_{b}(w)/\tilde{r}_{b}(w)$ is increasing and convex and $w^2[r_{b}(w)/\tilde{r}_{b}(w)]''$ is decreasing in $w>0$.
\end{theorem}
\begin{proof}
	On differentiating (\ref{eq3.17}) with respect to $\lambda_{i}$, we have
	\begin{align}
	\frac{\partial r_{X_{2:n}}(x)}{\partial \lambda_{i}}=&-\frac{1}{{\theta_i}^2}r^{'}_{b}\left(\frac{x-\lambda_i}{\theta_i}\right)+
	 \left[\frac{1}{{\theta_{i}}^2}g^{''}\left(\frac{x-\lambda_{i}}{\theta_{i}}\right)\right]\left[\sum\limits_{i=1}^{n}g\left(\frac{x-\lambda_{i}}{\theta_{i}}\right)+1\right]^{-1}\nonumber\\
	 &-\left[\frac{1}{\theta_{i}}g^{'}\left(\frac{x-\lambda_{i}}{\theta_{i}}\right)\sum\limits_{i=1}^{n}\frac{1}{\theta_{i}}g^{'}\left(\frac{x-\lambda_{i}}{\theta_{i}}\right)\right]
	\left[\sum\limits_{i=1}^{n}g\left(\frac{x-\lambda_{i}}{\theta_{i}}\right)+1\right]^{-2},
	\end{align}
	where $g\left(\frac{x-\lambda_{i}}
	{\theta}\right)=r_{b}\left(\frac{x-\lambda_{i}}
	{\theta}\right)/\tilde{r}_{b}\left(\frac{x-\lambda_{i}}
	{\theta}\right).$ Now, consider
	\begin{align}\label{eq3.15}
	\frac{\partial r_{X_{2:n}}(x)}{\partial \lambda_i}-\frac{\partial r_{X_{2:n}}(x)}{\partial \lambda_j}\overset{sign}{=} T_1+T_2+T_3,
	\end{align}
	where
	\begin{eqnarray*}
	T_1&=&	 \left[-\frac{1}{{\theta_i}^2}{r}_{b}'\left(\frac{x-\lambda_i}{\theta_i}\right)+\frac{1}{{\theta_j}^2}{r}_{b}'\left(\frac{x-\lambda_j}{\theta_j}\right)\right]\left[\sum\limits_{i=1}^{n}g\left(\frac{x-\lambda_{i}}{\theta_{i}}\right)+1\right]^2,\\
	 T_2&=&\left[\frac{1}{(x-\lambda_{i})^2}\left(\frac{x-\lambda_{i}}{\theta_{i}}\right)^2g^{''}\left(\frac{x-\lambda_{i}}{\theta_{i}}\right)-\frac{1}{(x-\lambda_{j})^2}\left(\frac{x-\lambda_{j}}{\theta_{j}}\right)^2g^{''}\left(\frac{x-\lambda_{j}}{\theta_{j}}\right)\right]
	\\
	&~&\times\left[\sum\limits_{i=1}^{n}g\left(\frac{x-\lambda_{i}}{\theta_{i}}\right)
	+1\right]~\mbox{and}\\
	 T_3&=&\left[-\frac{1}{\theta_{i}}g^{'}\left(\frac{x-\lambda_{i}}{\theta_{i}}\right)+\frac{1}{\theta_{j}}g^{'}\left(\frac{x-\lambda_{j}}{\theta_{j}}\right)\right]\left[\sum\limits_{i=1}^{n}\frac{1}{\theta_{i}}g^{'}\left(\frac{x-\lambda_{i}}{\theta_{i}}\right)\right].
	\end{eqnarray*}
	Further, under the assumptions made, it can be shown that
	$$	\frac{\partial r_{X_{2:n}}(x)}{\partial \lambda_i}-\frac{\partial r_{X_{2:n}}(x)}{\partial \lambda_j}\geq(\leq) ~0 ,$$
	for $\bm{\lambda}\in\mathcal{D}_+~(\mbox{or}~\mathcal{I}_+).$ Thus, $r_{X_{2:n}}(x)$ is Schur-convex with respect to $\boldsymbol{\lambda}\in\mathcal{D}_+~(\mbox{or}~\mathcal{I}_+).$ Hence, the desired result readily follows.
\end{proof}

Next, we assume that the location parameters are equal and fixed, and the scale parameter vectors are different. The stated result provides sufficient conditions, under which we get a better fail-safe system in the sense of the instantaneous failure rate. Particularly, it states that the system with majorized reciprocal of the scale parameter vector yields a system with smaller hazard rate function.
\begin{theorem}\label{th3.7}
	Let  $\boldsymbol{X}\sim \mathbb{ELS}(
	\boldsymbol{\lambda},\boldsymbol{\theta},\boldsymbol{\alpha};{F_{b}})$ and $\boldsymbol{Y}\sim \mathbb{ELS}(
	\boldsymbol{\mu},\boldsymbol{\delta},\boldsymbol{\beta};F_{b}),$ with $\boldsymbol{\lambda}=\boldsymbol{\mu}=\lambda \bm{1}_{n}$ and $\boldsymbol{\alpha}=\boldsymbol{\beta}=\bm{1}_n$. Further, consider $ \boldsymbol{\theta},~\boldsymbol{\delta}\in\mathcal{D}_+~(or~\mathcal{I}_+)$. Then,
	$1/{\boldsymbol\theta}\succeq^{m}1/{\boldsymbol\delta}\Rightarrow X_{2:n}\geq_{hr}Y_{2:n}$, provided $wr_{b}(w)$, $[r_{b}(w)/\tilde{r}_{b}(w)]$ are increasing, concave and $w[r_{b}(w)/\tilde{r}_{b}(w)]'$ is convex with respect to $w>0.$
\end{theorem}
\begin{proof}
	Rewrite the hazard rate function of $X_{2:n}$ as
	\begin{eqnarray}\label{eq3.24}
	 r_{X_{2:n}}(x)&=&\Phi(\boldsymbol{m})\nonumber\\
&=&\left[\sum\limits_{i=1}^{n}m_{i}r_{b}
\left((x-\lambda)m_{i}\right)\right]-\left[\sum\limits_{i=1}^{n}m_{i}g^{'}
\left((x-\lambda)m_{i}\right)\right]\left[\sum\limits_{i=1}^{n}g\left((x-\lambda)m_{i}\right)+1\right]^{-1},\nonumber\\
	\end{eqnarray}
	where $g\left((x-\lambda)m_i
	\right)=r_{b}\left((x-\lambda)m_i\right)/\tilde{r}_{b}\left((x-\lambda)m_i\right)$, $m_{i}=1/\theta_i,~i=1,\cdots,n$ and $\bm{m}=(m_1,\cdots,m_n).$
	After differentiating Equation \eqref{eq3.24} partially with respect to $m_{i}$, we obtain
	\begin{align}\label{eq3.25}
	\frac{\partial \Phi(\boldsymbol{m})}{\partial m_{i}}\overset{sign}{=}&\left[\frac{d}{dw}\left[wr_{b}(w)\right]_{w=\left((x-\lambda)m_{i}\right)}\right]
\left[\sum\limits_{i=1}^{n}g\left((x-\lambda)m_{i}\right)+1\right]^2-\left[\frac{d}{dw}\left[wg'(w)
\right]_{w=\left((x-\lambda)m_{i}\right)}\right]\nonumber\\
	 \times&\left[\sum\limits_{i=1}^{n}g\left((x-\lambda)m_{i}\right)+1\right]+
\left[[(x-\lambda)g^{'}(w)]_{w=\left((x-\lambda){m_{i}}\right)}\right]\left[
\sum\limits_{i=1}^{n}m_{i}g^{'}\left((x-\lambda_{i}){m_{i}}\right)\right].
	\end{align}
	To prove the result, it is sufficient to show that  $\Phi(\boldsymbol{m})$ is Schur-concave with respect to $\boldsymbol{m}\in\mathcal{I}_+~(or~\mathcal{D}_+).$ Using \eqref{eq3.25}, we have
	\begin{align}\label{eq3.26}
	\frac{\partial \Phi(\boldsymbol{m})}{\partial m_i}-\frac{\partial \Phi(\boldsymbol{m})}{\partial m_j}&\overset{sign}{=} T_1 + T_2 + T_3,
	\end{align}
	where
	\begin{eqnarray*}
	T_1&=& \left[\frac{d}{dw}\left[wr_{b}(w)\right]_{w=\left((x-\lambda)m_{i}\right)}-\frac{d}{dw}\left[wr_{b}(w)
\right]_{w=\left((x-\lambda)m_{j}\right)}\right]\left[\sum\limits_{i=1}^{n}g\left((x-\lambda){m_{i}}
\right)+1\right]^2,\\
	T_2 &=& \left[-\frac{d}{dw}\left[wg'(w)\right]_{w=\left((x-\lambda)m_{i}\right)}+\frac{d}{dw}\left[wg'(w)\right]_{w
=\left((x-\lambda)m_{j}\right)}\right]\left[\sum\limits_{i=1}^{n}g\left((x-\lambda)){m_{i}}\right)+1\right]~\mbox{and}\\
	T_3 &=& \left[(x-\lambda)[g^{'}(w)]_{w=\left((x-\lambda){m_{i}}\right)}-(x-\lambda)[g^{'}(w)]_{w=\left((x-\lambda){m_{j}}\right)}\right]\left[\sum\limits_{i=1}^{n}m_{i}[g^{'}(w)]_{w=\left((x-\lambda){m_{i}}\right)}\right].
	\end{eqnarray*}
	Now, applying the given assumptions and then, after some simplification, we obtain
	$$	\frac{\partial \Phi(\boldsymbol{m})}{\partial m_i}-\frac{\partial \Phi(\boldsymbol{m})}{\partial m_j}\geq(\leq)~0, \text{ for each } m_i\in\mathcal{I}_+~(or~\mathcal{D}_+).$$
	Thus, the result follows from Lemma \ref{lem2.2} (\ref{lem2.1})  and Definition \ref{def2.3}. This completes the proof of the theorem.
\end{proof}
\begin{remark}
	Let $F_b(x)=\frac{x-1}{x+1},~x\geq1$ be the cumulative distribution function of the baseline distribution. It can be shown that for this baseline distribution, $wr_{b}(w)$ and $[r_{b}(w)/\tilde{r}_{b}(w)]$ are increasing, concave and $w[r_{b}(w)/\tilde{r}_{b}(w)]'$ is convex  with respect to $w>0$.
\end{remark}
The following corollary is a direct consequence of the result given in the immediately preceding theorem.
\begin{corollary}\label{col3.6}
	Suppose  $\boldsymbol{X}\sim \mathbb{ELS}(
	\lambda,\boldsymbol{\theta},\boldsymbol{\alpha};F_{b})$ and $\boldsymbol{Y}\sim \mathbb{ELS}(
	\lambda,\theta,\boldsymbol{\beta};F_{b}),$ with $\boldsymbol{\alpha}=\boldsymbol{\beta}=\bm{1}_n$. Also, assume $  \boldsymbol{\theta}\in\mathcal{I}_+~(or~\mathcal{D}_+).$ Then,
	$1/ \theta= \sum_{i=1}^{n}1/{(n\theta_i)}\Rightarrow X_{2:n}\geq_{hr}Y_{2:n}$, provided $wr_{b}(w)$ and $[r_{b}(w)/\tilde{r}_{b}(w)]$ are increasing, concave and $w[r_{b}(w)/\tilde{r}_{b}(w)]'$ is convex  with respect to $w>0$.
	
\end{corollary}
Till, we have established ordering results between the second-order statistics arising from two independent batches of independent heterogeneous random observations. However, there are many practical situations, where the components of a system may have a structural dependence due to various reasons. The components may be manufactured by the same company, or by different companies but using similar technology. This results in a set of statistically dependent observations. In the following subsection, we consider that the collection of random observations are interdependent.
\subsection{Interdependent case}
In this subsection, we consider two fail-safe systems with heterogeneous exponentiated location-scale family distributed components. The additional assumption is that the components are dependent. The dependency is structured by Archimedean copula with generator $\psi$. Consider two $n$-dimensional random vectors $\bm{X}$ and $\bm{Y}$, where the $i$-th components of each vectors follow exponentiated location-scale models. Notationally, $X_{i}\sim \mathbb{ELS}(\lambda_i,\theta_i,\alpha_i;F_b,\psi)$ and  $Y_{i}\sim \mathbb{ELS}(\mu_i,\delta_i,\beta_i;F_b,\psi)$, for $i=1,\cdots,n.$ In terms of the vector notation, $\bm{X}\sim \mathbb{ELS}(\bm{\lambda},\bm{\theta},\bm{\alpha};F_b,\psi)$ and  $\bm{Y}\sim \mathbb{ELS}(\bm{\mu},\bm{\delta},\bm{\beta};F_b,\psi)$. Under the general set-up, the survival functions of $X_{2:n}$ and $Y_{2:n}$ are respectively obtained as
\begin{align}\label{eq3.27}
\small\bar{F}_{X_{2:n}}(x)=&\sum\limits_{l=1}^{n}\psi\left[\sum\limits_{k\neq l}^{n}\phi\left\{1-\left[F_{b}\left(\frac{x-\lambda_k}{\theta_k}\right)\right]^{\alpha_k}\right\}\right]-(n-1)\psi\left[\sum\limits_{k=1}^{n}\phi\left\{1-\left[F_{b}\left(\frac{x-\lambda_k}{\theta_k}\right)\right]^{\alpha_k}\right\}\right],\\
\small\bar{G}_{Y_{2:n}}(y)=&\sum\limits_{l=1}^{n}\psi\left[\sum\limits_{k\neq l}^{n}\phi\left\{1-\left[F_{b}\left(\frac{y-\mu_k}{\delta_k}\right)\right]^{\beta_k}\right\}\right]-(n-1)\psi\left[\sum\limits_{k=1}^{n}\phi\left\{1-\left[F_{b}\left(\frac{y-\mu_k}{\delta_k}\right)\right]^{\beta_k}\right\}\right],
\end{align}
where $x>\max\{\lambda_1,\cdots,\lambda_n \}$  and $y>\max\{\mu_1,\cdots,\mu_n\}.$ The first theorem in this subsection shows that the usual stochastic order exists between the second-order statistics under sub weak majorization order of the location parameter vectors. The proof is obtained using Lemmas \ref{lem2.1}, \ref{lem2.2} and Theorem $A.8$ of \cite{Marshall2011}.

\begin{theorem}\label{th3.8}
	Let  $\boldsymbol{X}\sim \mathbb{ELS}(
	\boldsymbol{\lambda},\boldsymbol{\theta},\boldsymbol{\alpha};F_{b},\psi)$ and $\boldsymbol{Y}\sim \mathbb{ELS}(
	\boldsymbol{\mu},\boldsymbol{\delta},\boldsymbol{\beta};F_{b},\psi),$ with $\boldsymbol{\alpha}=\boldsymbol{\beta}=\alpha \bm{1}_n\leq 1$ and $\boldsymbol{\theta}=\boldsymbol{\delta}$. Assume $\boldsymbol{\lambda},~ \boldsymbol{\theta},~\boldsymbol{\mu}\in\mathcal{I}_+~(or~\mathcal{D}_+)$ and $w r_{b}(w)$  is decreasing in $w>0$. Then, $\boldsymbol{\lambda}\succeq_{w}\boldsymbol{\mu}\Rightarrow X_{2:n}\geq_{st}Y_{2:n}$, provided $\psi$ is log-concave.
\end{theorem}
\begin{proof}
	 Denote
$\Psi_1(\boldsymbol{\lambda})=\bar{F}_{X_{2:n}}(x),$ where $\alpha_k$ is replaced by $\alpha$ in (\ref{eq3.27}). Here, we need to show that $\Psi_1(\boldsymbol{\lambda})$ is increasing and  Schur-convex with respect to $\boldsymbol{\lambda}\in \mathcal{I}_+~(or~\mathcal{D}_+).$
For convenience, we further denote  ${t_i=\phi[1-[F_{b}(\frac{x-\lambda_i}{\theta_i})]^{\alpha}]}.$
Differentiating $\Psi_1(\boldsymbol{\lambda})=\bar{F}_{X_{2:n}}(x),$ where $\bar{F}_{X_{2:n}}(x)$ is given by \eqref{eq3.27} with respect to $\lambda_i,$ for $i=1,\cdots,n$, we have
\begin{equation}\label{eq4}
\frac{\partial\Psi_1(\boldsymbol{\lambda})}{\partial\lambda_i}=\varrho_{1i}(\bm{\lambda})\varrho_1(\lambda_i),
\end{equation}
where
\begin{eqnarray*}
\varrho_{1i}(\bm{\lambda})&=&(n-1)\psi'\left(\sum\limits_{k=1}^{n}t_k\right)-\sum\limits_{l\notin\{i,j\}}^{n}\psi'\left(\sum\limits_{k\neq l}^{n}t_k\right)
-\psi'\left(t_i+\sum\limits_{k\notin\{i,j\}}^{n}\left\{t_k\right\}\right)~\mbox{and}
\\
\varrho_1(\lambda_i)&=&-\left[\frac{{\psi}({t_i})}{{\psi}'\left(t_i\right) }\right]\left[\frac{\alpha t^{\alpha-1}(1-t)}{1-t^{\alpha}}\right]_{t=F_{b}\left(\frac{x-\lambda_i}{\theta_i}\right)}\left[\frac{[wr_{b}(w)]_{w=\left(\frac{x-\lambda_i}{\theta_i}\right)}}{x-\lambda_i}\right].
\end{eqnarray*}
Let the vectors belong to $\mathcal{I}_+$. Then, for $1\leq i\leq j \leq n,$ we have $\theta_i\leq\theta_j$ and $\lambda_i\leq\lambda_j.$ Thus, we obtain $\frac{x-\lambda_i}{\theta_i}\geq\frac{x-\lambda_j}{\theta_j}$ and $[F_{b}(\frac{x-\lambda_i}{\theta_i})]^{\alpha}\geq[F_{b}(\frac{x-\lambda_j}{\theta_j})]^{\alpha}$.  Hence, $t_i\geq t_j$, for all $i=1,\cdots,n.$ Now, for $l=1,\cdots,n$, by the properties of the generator of an Archimedean copula, we have
\begin{eqnarray}
\psi'\left(\sum\limits_{k=1}^{n}t_k\right)&\geq& \psi'\left(\sum\limits_{k\neq l}^{n}t_k\right)\nonumber\\
\Rightarrow (n-1)\psi'\left(\sum\limits_{k=1}^{n}t_k\right)&\geq& \sum\limits_{l\neq i}^{n}\psi'\left(\sum\limits_{k\neq l}^{n}t_k\right),
\end{eqnarray}
where $i=1,\cdots,n.$
 Thus,
 \begin{align}
 \varrho_{1i}(\bm{\lambda})&=(n-1)\psi'\left(\sum\limits_{k=1}^{n}t_k\right)-\sum\limits_{l\notin\{i,j\}}^{n}\psi'\left(\sum\limits_{k\neq l}^{n}t_k\right)
 -\psi'\left(\sum\limits_{k\neq j}^{n}\left\{t_k\right\}\right)\nonumber\\
 &=(n-1)\psi'\left(\sum\limits_{k=1}^{n}t_k\right)-\sum\limits_{l\neq i}^{n}\psi'\left(\sum\limits_{k\neq l}^{n}t_k\right)\nonumber\\&\geq 0.
 \end{align}
Furthermore, $\varrho_1(\lambda_i)$ is positive, since  $\psi(x)$ is decreasing. We will show that  both $\varrho_{1i}(\bm{\lambda})$ and $ \varrho_1(\lambda_i)$ are increasing with respect to $\lambda_i,$ for $i=1,\cdots,n.$ Note that $\psi(x)$ is decreasing and convex. So, $\psi'(x)$ is negative and increasing. Again, $\phi[1-[F_{b}(\frac{x-\lambda_i}{\theta_i})]^{\alpha}]$ is decreasing in $\lambda_i,$ for $i=1,\cdots,n.$ Thus, $\varrho_{1i}(\bm{\lambda})$ is increasing in $\lambda_i,$ for all $i= 1,\cdots,n.$ Under the assumptions made and Lemma \ref{lem2.4}, we obtain the following three inequalities:
\begin{eqnarray*}\label{eq3.32}
-\left[\frac{{\psi}({t_i})}{{\psi}'\left(t_i\right) }\right]&\leq&	  -\left[\frac{{\psi}({t_j})}{{\psi}'\left(t_j\right) }\right], \\		\left[\frac{\alpha t^{\alpha-1}(1-t)}{1-t^{\alpha}}\right]_{t=F_{b}\left(\frac{x-\lambda_i}{\theta_i}\right)}&\leq&\left[\frac{\alpha t^{\alpha-1}(1-t)}{1-t^{\alpha}}\right]_{t=F_{b}\left(\frac{x-\lambda_j}{\theta_j}\right)},\\
\left[w r_{b}(w)\right]_{w=\left(\frac{x-\lambda_i}{\theta_i}\right)}&\leq&\left[w r_{b}(w)\right]_{w=\left(\frac{x-\lambda_j}{\theta_j}\right)}.
\end{eqnarray*}
Making use of these inequities, it can be shown that $\varrho_1(\lambda_i)$ is increasing  with respect to $\lambda_i,$ for all $i=1,\cdots,n.$ Thus,  $\Psi_1(\boldsymbol{\lambda})$ is increasing and  Schur-convex with respect to $\boldsymbol{\lambda}\in \mathcal{I}_+.$ Rest of the proof follows from Theorem $A.8$ of \cite{Marshall2011}. We note that the proof for the other case is similar and thus, it is omitted for the sake of conciseness.
\end{proof}

In the previous theorem, we assume that the scale parameters of the model are equal and vector-valued. Next, we consider $\boldsymbol{\theta}=\boldsymbol{\delta}=\theta \bm{1}_{n}$ and obtain similar result. The proof is omitted.
\begin{theorem}\label{th3.8*}
Under the set-up as in Theorem \ref{th3.8}, we further assume that $\boldsymbol{\theta}=\boldsymbol{\delta}=\theta \bm{1}_{n}$. Let $\boldsymbol{\lambda},~\boldsymbol{\mu}\in\mathcal{I}_+~(or~\mathcal{D_+})$ and $ r_{b}(w)$ be decreasing in $w>0$. Then, $\boldsymbol{\lambda}\succeq_{w}\boldsymbol{\mu}\Rightarrow X_{2:n}\geq_{st}Y_{2:n}$, provided $\psi$ is log-concave.
\end{theorem}

In order to illustrate Theorem \ref{th3.8*}, we consider the following example. It is noted that there are lot of distributions with decreasing hazard rate function.
\begin{example}
	Suppose $\{X_{1},X_{2},X_{3}\}$ and
	$\{Y_{1},Y_{2},Y_{3}\}$ are two sets of dependent random variables such that $X_{i}\sim	 F_{b}^{\alpha}(\frac{x-\lambda_{i}}{\theta})$ and $Y_{i}\sim
	F_{b}^{\alpha}(\frac{x-\mu_i}{\theta})$, for $i=1,2,3,$ with a common generator $\psi.$ Let $\psi(x)=e^{(1-e^x)/a}, ~x>0,~0<a\leq1.$ Here, $\psi$ can be easily shown to be log-concave. Consider the baseline distribution function as exponentiated Weibull distribution, where $F_{b}(x)=(1-\exp(-x^d))^c, ~x,~d,~c>0$. Note that for $d\leq1$ and $dc\leq1$, the hazard rate function of the baseline distribution is decreasing. Set $\boldsymbol{\lambda}=(4,6,7)$,
	$\boldsymbol{\mu}=(2.5,3.5,4)$,
	$\theta=5,~a=0.5,~ \alpha=0.1,~d=0.5$ and $c=0.2$. Therefore, clearly, $\boldsymbol{\lambda}\succeq_{w}\boldsymbol{\mu}$ and $\boldsymbol{\lambda},~\boldsymbol{\mu}\in\mathcal{I}_{+}$.
	Thus, as an application of Theorem \ref{th3.8*},
	we have $X_{2:3}\ge_{st}Y_{2:3}.$
\end{example}

Below, in the consecutive theorems, we consider different sets of scale parameters. It is shown that under some conditions, the reciprocally majorization order or weak super majorization order between the reciprocal of the scale parameter vectors implies the usual stochastic order between the second-order statistics.
\begin{theorem}\label{th3.9}
	Assume that  $\boldsymbol{X}\sim \mathbb{ELS}(
	\boldsymbol{\lambda},\boldsymbol{\theta},\boldsymbol{\alpha};F_{b},\psi)$ and $\boldsymbol{Y}\sim \mathbb{ELS}(
	\boldsymbol{\mu
	},\boldsymbol{\delta},\boldsymbol{\beta};F_{b},\psi),$ with $\boldsymbol{\lambda}=\boldsymbol{\mu}$ and $\boldsymbol{\alpha}=\boldsymbol{\beta}=\alpha \bm{1}_{n}\leq1$. Also, let $\boldsymbol{\lambda},~ \boldsymbol{\theta},~\boldsymbol{\delta}\in\mathcal{I}_+~(or~\mathcal{D}_+)$ and $w^2 r_{b}(w)$ be decreasing in $w>0$. Then, ${\frac{1}{\boldsymbol\theta}}\succeq^{rm}{\frac{1}{\boldsymbol\delta}}\Rightarrow X_{2:n}\geq_{st}Y_{2:n}$, provided $\psi$ is log-concave.
\end{theorem}
\begin{proof}
	Let us denote
	$\Psi_2\left({\frac{1}{\boldsymbol\theta}}\right)=\bar{F}_{X_{2:n}}(x)$ and
	 ${t_i=\phi[1-[F_{b}(\frac{x-\lambda_i}{\theta_i})]^{\alpha}]},$ for $i=1,\cdots,n.$ Here, the proof is presented for $\boldsymbol{\lambda},~\bm{\theta}\in \mathcal{I}_+.$ The proof for other case is analogous.  The partial derivative of $\Psi_2\left({\frac{1}{\boldsymbol\theta}}\right)$ with respect to $\theta_i$, for $i=1,\cdots,n$, is obtained as
	 \begin{equation}\label{eqd6}
	 \frac{\partial\Psi_2\left({\frac{1}{\boldsymbol\theta}}\right)}{\partial\theta_i}=
	 \varrho_{2i}\left(\frac{1}{\bm{\theta}}\right)\varrho_2\left(\frac{1}{\theta_i}\right),
	 \end{equation}
	 where
	 \begin{eqnarray}\label{eqd7}
	 \varrho_{2i}\left(\frac{1}{\bm{\theta}}\right)&=&(n-1)\psi'\left(\sum\limits_{k=1}^{n}t_k\right)-\sum\limits_{l\notin\{i,j\}}^{n}\psi'\left(\sum\limits_{k\neq l}^{n}t_k\right)
	 -\psi'\left(t_i+\sum\limits_{k\notin\{i,j\}}^{n}\left\{t_k\right\}\right),~\mbox{and}\\
	 \varrho_2\left(\frac{1}{\theta_i}\right)&=&-\left(\frac{\alpha}{x-\lambda_i}\right)[w^2{{r}_b}(w)]_{w=(\frac{x-\lambda_i}{\theta_i})}\left[\frac{{\psi}({t_i})}{{\psi}'\left(t_i\right) }\right]\left[\frac{\alpha t^{\alpha-1}(1-t)}{1-t^{\alpha}}\right]_{t=F_{b}\left(\frac{x-\lambda_i}{\theta_i}\right)}.
	 \end{eqnarray}
	 To establish the stated result, we need to show that $\Psi_2\left({\frac{1}{\boldsymbol\theta}}\right)$ is increasing and  Schur-convex with respect to $\boldsymbol{\theta}\in \mathcal{I}_+$. This follows similarly to the proof of Theorem \ref{th3.8}. Thus, it is omitted.
\end{proof}

\begin{theorem}\label{th3.10}
	Suppose  $\boldsymbol{X}\sim \mathbb{ELS}(
	\boldsymbol{\lambda},\boldsymbol{\theta},\boldsymbol{\alpha};F_{b},\psi)$ and $\boldsymbol{Y}\sim \mathbb{ELS}(
	\boldsymbol{\mu},\boldsymbol{\delta},\boldsymbol{\beta};F_{b},\psi),$ with $\boldsymbol{\lambda}=\boldsymbol{\mu}$ and $\boldsymbol{\alpha}=\boldsymbol{\beta}=\alpha \bm{1}_{n}\leq1$. Further, let $\boldsymbol{\lambda},~ \boldsymbol{\theta},~\boldsymbol{\delta}\in\mathcal{I}_+~(or~\mathcal{D}_+)$ and $w r_{b}(w)$ be decreasing in $w>0$. Then, ${\frac{1}{\boldsymbol\theta}}\succeq^{w}{\frac{1}{\boldsymbol\delta}}\Rightarrow X_{2:n}\geq_{st}Y_{2:n}$, provided $\psi$ is log-concave.
\end{theorem}
\begin{proof}
	To prove the stated result, first we define
	\begin{eqnarray}\label{eq3.35}
	\Psi_3\left({{\boldsymbol v}}\right)&=&\bar{F}_{X_{2:n}}(x)\nonumber\\
	&=&\sum\limits_{l=1}^{n}\psi\left[\sum\limits_{k\neq l}^{n}\phi\left\{1-\left[F_{b}\left(({x-\lambda_k}){v_k}\right)\right]^{\alpha}\right\}\right]-(n-1)\psi\left[\sum\limits_{k=1}^{n}\phi\left\{1-\left[F_{b}\left({(x-\lambda_k)}{v_k}\right)\right]^{\alpha}\right\}\right],\nonumber\\
	\end{eqnarray}
	where $v_i=1/\theta_i$, for $i=1,\cdots,n.$ Denote $t_i=\phi\left[1-\left[F_{b}\left(({x-\lambda_i}){v_i}\right)\right]^{\alpha}\right]$ for simplicity of the presentation. On differentiating (\ref{eq3.35}) partially with respect to $v_i$, for $i=1,\cdots,n$, we obtain
	\begin{equation}\label{eq14}
	\frac{\partial\Psi_3\left({\boldsymbol v}\right)}{\partial v_i}=\varrho_{3i}(\bm{v})\varrho_3(v_i),
	\end{equation}
	where
	\begin{eqnarray}	 \varrho_{3i}(\bm{v})&=&(n-1)\psi'\left(\sum\limits_{k=1}^{n}t_k\right)-\sum\limits_{l\notin\{i,j\}}^{n}\psi'\left(\sum\limits_{k\neq l}^{n}t_k\right)
	-\psi'\left(t_i+\sum\limits_{k\notin\{i,j\}}^{n}\left\{t_k\right\}\right)~\mbox{and}
	\\
	 \varrho_3(v_i)&=&[w{r_b}(w)]_{w=(\left({(x-\lambda_i)}{v_i}\right))}\left[\frac{{\psi}({t_i})}{{\psi}'\left(t_i\right) }\right]\left[\frac{\alpha t^{\alpha-1}(1-t)}{1-t^{\alpha}}\right]_{t=F_{b}\left({(x-\lambda_i)}{v_i}\right)}.
	\end{eqnarray}
	The required result readily follows, if we show that $\Psi_3(\boldsymbol{v})$ is decreasing and  Schur-convex with respect to $\boldsymbol{v}\in \mathcal{D}_+~(or~\mathcal{I}_+).$ We omit the remaining steps of the proof, since these are similar to that of Theorem \ref{th3.8}.
\end{proof}

In the next theorem, we show that under the restriction $\boldsymbol{\lambda}=\boldsymbol{\mu}=\lambda\bm{1}_{n}$, the usual stochastic order between $X_{2:n}$ and $Y_{2:n}$ exists when the reciprocal of the scale parameter vectors are connected with the weak super majorization order. The proof is similar to that of Theorem \ref{th3.10}. Thus, we only present the statement of the result.
\begin{theorem}\label{th3.14}
	Suppose  $\boldsymbol{X}\sim \mathbb{ELS}(
	\boldsymbol{\lambda},\boldsymbol{\theta},\boldsymbol{\alpha};F_{b},\psi)$ and $\boldsymbol{Y}\sim \mathbb{ELS}(
	\boldsymbol{\mu},\boldsymbol{\delta},\boldsymbol{\beta};F_{b},\psi),$ with $\boldsymbol{\lambda}=\boldsymbol{\mu}=\lambda\bm{1}_{n}$ and $\boldsymbol{\alpha}=\boldsymbol{\beta}=\alpha\bm{1}_{n}\leq1$. Let $ \boldsymbol{\theta},~\boldsymbol{\delta}\in\mathcal{I}_{+}~(or~\mathcal{D}_{+})$ and $ r_{b}(w)$ be decreasing in $w$. Then, ${\frac{1}{\boldsymbol\theta}}\succeq^{w}{\frac{1}{\boldsymbol\delta}}\Rightarrow X_{2:n}\geq_{st}Y_{2:n}$, provided $\psi$ is log-concave.
\end{theorem}
The numerical example given below,  illustrates Theorem \ref{th3.14}.
\begin{example}\label{ex3.5}
	Let $\{X_{1},X_{2},X_{3}\}$ and
	$\{Y_{1},Y_{2},Y_{3}\}$ be two sets of dependent random variables such that $X_{i}\sim	 F^{\alpha}_b(\frac{x-\lambda}{\theta_i})$ and $Y_{i}\sim
	F^{\alpha}_b(\frac{x-\lambda}{\theta_i})$, for $i=1,2,3.$ Let $\psi(x)=e^{-x},~x>0.$ Here, $\psi$ is log-concave. Take the baseline distribution function as lower-truncated Weibull distribution with $F_{b}(x)=1-\exp(1-x^a),~x\geq1,~a>0$. The hazard rate function $r_b(x)$ of this distribution is decreasing for $a\leq 1.$ Set $\boldsymbol{\theta}=(4.5,6.5,7.5)$,
	$\boldsymbol{\delta}=(2.5,3.5,4)$,
	$\lambda=5$, $a=0.12$ and $\alpha=0.9$. For these values, it can be verified that $1/\boldsymbol{\theta}\succeq^{w}1/\boldsymbol{\delta}$. Further, clearly,  $\boldsymbol{\theta},~\boldsymbol{\delta}\in\mathcal{I}_{+}$.
	Therefore, Theorem \ref{th3.14} provides
	$X_{2:3}\ge_{st}Y_{2:3},$ which can be verified from Figure $1b$.
\end{example}

In the above theorems of this subsection, we have taken a common Archimedean copula to describe the dependence among the random variables. It is then natural to ask if the result holds for the case of different Archimedean copulas. In this part of the paper, we will search the answers of this question. First, we present the following lemmas, which will be useful in this sequel. The first lemma can be found in \cite{li2015ordering}.
\begin{lemma}\label{lem3.1}
	For two $n$-dimensional Archimedean copulas $C_{\psi_1}$ and  $C_{\psi_2}$, if $\phi_2\circ\psi_1$ is super-additive, then $C_{\psi_1}(\boldsymbol{v})\leq C_{\psi_2}(\boldsymbol{v})$, for all $\boldsymbol{v}\in[0,1]^n.$ A function $f$ is said to be super-additive, if $f(x+y)\ge f(x)+f(y),$ for all $x$ and $y$ in the domain of $f.$
\end{lemma}

\begin{lemma}\label{lem3.2}
	 For two $n$-dimensional Archimedean copulas $C_{\psi_1}$ and  $C_{\psi_2}$, if $\phi_2\circ\psi_1$ is sub-additive, then $C_{\psi_1}(\boldsymbol{v})\geq C_{\psi_2}(\boldsymbol{v})$, for all $\boldsymbol{v}\in[0,1]^n.$ A function $f$ is said to be sub-additive, if $f(x+y)\le f(x)+f(y),$ for all $x$ and $y$ in the domain of $f.$
\end{lemma}
\begin{proof}
The proof of the theorem can be done by similar approach as in Theorem 4.4.2 of \cite{nelsen2006introduction}. Therefore, it is omitted.
\end{proof}

In the following theorems, we obtain sufficient conditions, under which the usual stochastic order holds between the second-order statistics. It is assumed that  two collections of dependent observations have different dependent structures.

\begin{theorem}\label{th3.15}
	Let  $\boldsymbol{X}\sim \mathbb{ELS}(
	\boldsymbol{\lambda},\boldsymbol{\theta},\boldsymbol{\alpha};F_{b},\psi_1)$ and $\boldsymbol{Y}\sim \mathbb{ELS}(
	\boldsymbol{\mu},\boldsymbol{\delta},\boldsymbol{\beta};F_{b},\psi_2),$ with $\boldsymbol{\alpha}=\boldsymbol{\beta}=\alpha \bm{1}_n\leq 1$ and $\boldsymbol{\theta}=\boldsymbol{\delta}$. Let $\phi_1=\psi^{-1}_1$, $\phi_2=\psi^{-1}_2$ and  $\psi_1$ or $\psi_2$ is log-concave.
Also, assume $\boldsymbol{\lambda},~ \boldsymbol{\theta},~\boldsymbol{\mu}\in\mathcal{I}_+~(or~\mathcal{D_+})$ and $w r_{b}(w)$  is decreasing in $w$. Then,
 \begin{itemize}
 \item[(i)] $\boldsymbol{\lambda}\succeq_{w}\boldsymbol{\mu}\Rightarrow X_{2:n}\geq_{st}Y_{2:n}$, provided $\phi_2\circ\psi_1$ is sub-additive.
     \item[(ii)] $\boldsymbol{\mu}\succeq_{w}\boldsymbol{\lambda}\Rightarrow X_{2:n}\leq_{st}Y_{2:n}$, provided $\phi_2\circ\psi_1$ is super-additive.
 \end{itemize}
\end{theorem}
\begin{proof} $(i)$ To prove the first part, we denote
\begin{eqnarray}
A(x,\boldsymbol{ \lambda},\boldsymbol{ \theta}; F_{b},\psi_1)&=&\sum\limits_{l=1}^{n}\psi_1\left[\sum\limits_{k\neq l}^{n}\phi_1\left\{1-\left[F_{b}\left(\frac{x-\lambda_k}{\theta_k}\right)\right]^{\alpha}\right\}\right]\nonumber\\&~&-(n-1)\psi_1\left[\sum\limits_{k=1}^{n}\phi_1\left\{1-\left[F_{b}\left(\frac{x-\lambda_k}{\theta_k}\right)\right]^{\alpha}\right\}\right]
\end{eqnarray}
and
\begin{eqnarray}
B(x,\boldsymbol{ \mu},\boldsymbol{ \theta}; F_{b},\psi_2)&=&\sum\limits_{l=1}^{n}\psi_2\left[\sum\limits_{k\neq l}^{n}\phi_2\left\{1-\left[F_{b}\left(\frac{x-\mu_k}{\theta_k}\right)\right]^{\alpha}\right\}\right]\nonumber\\
&~&-(n-1)\psi_2\left[\sum\limits_{k=1}^{n}\phi_2\left\{1-\left[F_{b}\left(\frac{x-\mu_k}{\theta_k}\right)\right]^{\alpha}\right\}\right].
\end{eqnarray}
Applying sub-additive property of $\phi_2\circ\psi_1$ and Lemma \ref{lem3.2}, we obtain
\begin{eqnarray}\label{eqn1}
A(x,\boldsymbol{ \mu},\boldsymbol{ \theta}; F_{b},\psi_1)\geq B(x,\boldsymbol{ \mu},\boldsymbol{ \theta};F_{b},\psi_2).
\end{eqnarray}
Further, using Theorem \ref{th3.8}, we obtain that $\boldsymbol{ \lambda}\succeq_{w}\boldsymbol{\mu}$ implies
\begin{eqnarray}\label{eqn2}
A(x,\boldsymbol{ \lambda},\boldsymbol{ \theta}; F_{b},\psi_1)\geq A(x,\boldsymbol{ \mu},\boldsymbol{ \theta};F_{b},\psi_1).
\end{eqnarray}
Now, combining \eqref{eqn1} and \eqref{eqn2}, we get the required result. This completes the proof of the theorem. \\
\\
$(ii)$
	Denote
	\begin{eqnarray}
	A(x,\boldsymbol{ \lambda},\boldsymbol{ \theta}; F_{b},\psi_1)&=&\sum\limits_{l=1}^{n}\psi_1\left[\sum\limits_{\overset{k=1}{k\neq l}}^{n}\phi_1\left\{1-\left[F_{b}\left(\frac{x-\lambda_k}{\theta_k}\right)\right]^{\alpha}\right\}\right]\nonumber\\
	 &~&-(n-1)\psi_1\left[\sum\limits_{k=1}^{n}\phi_1\left\{1-\left[F_{b}\left(\frac{x-\lambda_k}{\theta_k}\right)\right]^{\alpha}\right\}\right]
	\end{eqnarray}
	and
	\begin{eqnarray}
	B(x,\boldsymbol{ \mu},\boldsymbol{ \theta}; F_{b},\psi_2)&=&\sum\limits_{l=1}^{n}\psi_2\left[\sum\limits_{\overset{k=1}{k\neq l}}^{n}\phi_2\left\{1-\left[F_{b}\left(\frac{x-\mu_k}{\theta_k}\right)\right]^{\alpha}\right\}\right]\nonumber\\&~&-(n-1)\psi_2\left[\sum\limits_{k=1}^{n}\phi_2\left\{1-\left[F_{b}\left(\frac{x-\mu_k}{\theta_k}\right)\right]^{\alpha}\right\}\right].
	\end{eqnarray}
	The super-additive property of $\phi_2\circ\psi_1$ and Lemma \ref{lem3.1} yield
	\begin{eqnarray}\label{eqn3}
	A(x,\boldsymbol{ \lambda},\boldsymbol{ \theta}; F_{b},\psi_1)\leq B(x,\boldsymbol{ \lambda},\boldsymbol{ \theta}; F_{b},\psi_2).
	\end{eqnarray}
	To prove the result, we have to establish
	\begin{eqnarray}\label{eqn4}
	B(x,\boldsymbol{ \mu},\boldsymbol{ \theta};F_{b},\psi_2)\geq B(x,\boldsymbol{ \lambda},\boldsymbol{ \theta};F_{b},\psi_2),
	\end{eqnarray}
	which is equivalent to show that $B(x,\boldsymbol{ \mu},\boldsymbol{ \theta};F_{b},\psi_2)$ is Schur-convex with respect to $\boldsymbol{ \mu}\in\mathcal{I_+}~ (\text{or } \mathcal{D_+}).$ This follows similarly from the proof of Theorem \ref{th3.8}. Hence, it is not presented.
\end{proof}

In the next result, we consider different scale parameters. We only present the statement, since the proof can be completed using similar arguments as in Theorem \ref{th3.15}.
\begin{theorem}\label{th3.16}
	Let  $\boldsymbol{X}\sim \mathbb{ELS}(
	\boldsymbol{\lambda},\boldsymbol{\theta},\boldsymbol{\alpha};F_{b},\psi_1)$ and $\boldsymbol{Y}\sim \mathbb{ELS}(
	\boldsymbol{\mu},\boldsymbol{\delta},\boldsymbol{\beta};F_{b},\psi_2),$ with $\boldsymbol{\alpha}=\boldsymbol{\beta}=\alpha \bm{1}_{n}\leq 1$ and $\boldsymbol{\lambda}=\boldsymbol{\mu}$. Let $\phi_1=\psi^{-1}_1$, $\phi_2=\psi^{-1}_2$ and $\psi_1$ or $\psi_2$ is log-concave. Further, assume $\boldsymbol{\lambda},~ \boldsymbol{\theta},~\boldsymbol{\delta}\in\mathcal{I}_{+}~(or~\mathcal{D}_{+})$ and $w r_{b}(w)$  is decreasing in $w$. Then,
 \begin{itemize}
 \item[(i)] ${\frac{1}{\boldsymbol\theta}}\succeq^{w}{\frac{1}{\boldsymbol\delta}}\Rightarrow X_{2:n}\geq_{st}Y_{2:n}$, provided $\phi_2\circ\psi_1$ is sub-additive.
 \item[(ii)] ${\frac{1}{\boldsymbol\delta}}\succeq^{w}{\frac{1}{\boldsymbol\theta}}\Rightarrow X_{2:n}\leq_{st}Y_{2:n}$, provided $\phi_2\circ\psi_1$ is super-additive.
 \end{itemize}
\end{theorem}

\begin{remark}
If we take $\boldsymbol{ \lambda}=(0,\cdots,0)$ and $1/\boldsymbol{ \theta}=\boldsymbol{ \lambda}$, $1/\boldsymbol{ \delta}=\boldsymbol{ \mu}$ and $\psi_1=\psi_2=\psi$, then the second part of Theorem  \ref{th3.16} reduces to Theorem $4.3$ of \cite{Li2016}.
\end{remark}
The following example is an explanation of the second part of Theorem  \ref{th3.16}.
\begin{example}
	Let $\{X_{1},X_{2},X_{3}\}$ and
	$\{Y_{1},Y_{2},Y_{3}\}$ be two sets of interdependent random variables such that $X_{i}\sim	 F_{b}^{\alpha}(\frac{x-\lambda_{i}}{\theta_i})$ and $Y_{i}\sim
	F_{b}^{\alpha}(\frac{x-\lambda_{i}}{\delta_i})$. The generators are taken as $\psi_1(x)=e^{(1-e^x)/a_1}$ and $\psi_2(x)=e^{(1-e^x)/a_2},$ where $x>0,~0<a_1,~a_2\leq1$. Here, $\psi_1$ and $\psi_2$ both are log-concave.  Also, it is clear that for $a_1\geq a_2>0,$ $\phi_2\circ\psi_1(x)=\log(1-\frac{a_2}{a_1}(1-e^x))$ is convex. This implies $\phi_2\circ\psi_1$ is super-additive. Let the baseline distribution function be the Pareto distribution as in Example \ref{ex3.1}. Clearly, $wr_b(w)$ is decreasing. Consider $\boldsymbol{\delta}=(6,7,8)$,
	$\boldsymbol{\theta}=(3,4,5)$,
	$\boldsymbol\lambda=(2,5,6),~ \alpha=0.8,~b=0.12,~a_1=0.5,$ and $a_2=0.1$. It can be checked that the conditions $1/\boldsymbol{\delta}\succeq^{w}1/\boldsymbol{\theta}$ and $\boldsymbol{\theta},~\boldsymbol{\delta}, ~\boldsymbol{\lambda}\in\mathcal{I}_{+}$ hold.
	Therefore, as an application of the second part of Theorem  \ref{th3.16},
	we have $X_{2:3}\le_{st}Y_{2:3}.$
\end{example}

In the next theorem, we consider that the scale parameters are equal and fixed. For the sake of brevity, we skip the proofs of the following two results.
\begin{theorem}\label{th3.17}
	Let  $\boldsymbol{X}\sim \mathbb{ELS}(
	\boldsymbol{\lambda},\boldsymbol{\theta},\boldsymbol{\alpha};F_{b},\psi_1)$ and $\boldsymbol{Y}\sim \mathbb{ELS}(
	\boldsymbol{\mu},\boldsymbol{\delta},\boldsymbol{\beta};F_{b},\psi_2),$ with $\boldsymbol{\alpha}=\boldsymbol{\beta}=\alpha \bm{1}_n\leq 1$ and $\boldsymbol{\theta}=\boldsymbol{\delta}=\theta\bm{1}_{n}$. Let $\phi_1=\psi^{-1}_1$, $\phi_2=\psi^{-1}_2$ and $\phi_2\circ\psi_1$ be super-additive. Also, assume $\boldsymbol{\lambda},~\boldsymbol{\mu}\in\mathcal{I}_{+}~(or~\mathcal{D}_{+})$ and $ r_{b}(w)$  is decreasing in $w$. Then, $\boldsymbol{\mu}\succeq_{w}\boldsymbol{\lambda}\Rightarrow X_{2:n}\leq_{st}Y_{2:n}$, provided $\psi_1$ or $\psi_2$ is log-concave.
\end{theorem}

In the following theorem, the location parameters are assumed to be equal and fixed. We only provide the statement of the result for the sake of brevity.

\begin{theorem}\label{th3.18}
	Let  $\boldsymbol{X}\sim \mathbb{ELS}(
	\boldsymbol{\lambda},\boldsymbol{\theta},\boldsymbol{\alpha};F_{b},\psi_1)$ and $\boldsymbol{Y}\sim \mathbb{ELS}(
	\boldsymbol{\mu},\boldsymbol{\delta},\boldsymbol{\beta};F_{b},\psi_2),$ with $\boldsymbol{\alpha}=\boldsymbol{\beta}=\alpha \bm{1}_n\leq 1$ and $\boldsymbol{\lambda}=\boldsymbol{\mu}=\lambda \bm{1}_n$. Let $\phi_1=\psi^{-1}_1$, $\phi_2=\psi^{-1}_2$ and $\phi_2\circ\psi_1$ be super-additive. Also, assume $ \boldsymbol{\theta},~\boldsymbol{\delta}\in\mathcal{I}_{+}~(or~\mathcal{D}_{+})$ and $ r_{b}(w)$  is decreasing in $w$. Then, ${\frac{1}{\boldsymbol\delta}}\succeq^{w}{\frac{1}{\boldsymbol\theta}}\Rightarrow X_{2:n}\leq_{st}Y_{2:n}$, provided $\psi_1$ or $\psi_2$ is log-concave.
\end{theorem}

We end this section with the following counterexamples. The first one shows that if some of the conditions of Theorem \ref{th3.17} get violated, then the usual stochastic order between the second-order statistics does not hold.
\begin{counterexample}\label{cex3.1}
	Consider the baseline distribution function as  $F_{b}(x)=1-\exp(1-x^a),~x\geq1,~a>0$. For this, when $0<a\leq 1,$ the hazard rate function is decreasing. Let $\{X_{1},X_{2},X_{3}\}$ and
	$\{Y_{1},Y_{2},Y_{3}\}$ be the collections of dependent random variables with $X_{i}\sim	 F_{b}^{\alpha}(\frac{x-\lambda_i}{\theta})$ and $Y_{i}\sim
	F_{b}^{\alpha}(\frac{x-\mu_i}{\theta})$, for $i=1,2,3$ with generators $\psi_1(x)=e^{\frac{1}{a_1}(1-e^x)}$ and $\psi_2(x)=e^{\frac{1}{a_2}(1-e^x)},$ respectively, where $x>0,~0<a_1,~a_2\leq1$. Here, $\psi_1$ and $\psi_2$ both are log-concave.  Assume $a_1=0.001$ and $a_2=0.5$. Then, $\phi_2\circ\psi_1(x)$ is sub-additive, and hence is not super-additive.  Set
	$\boldsymbol{\lambda}=(3,4,5)$,
	$\boldsymbol{\mu}=(6,7,8)$,
	$\theta=6,~ \alpha=8$ and $a=0.9$. Thus, all the conditions of Theorem \ref{th3.17} are satisfied except super-additivity of $\phi_2\circ\psi_1(x)$ and the restriction on $\alpha$. Now,
	we plote the graphs of $\bar{F}_{X_{2:3}}(x)$ and $\bar{F}_{Y_{2:3}}(x)$ in Figure $2a$. The graphs cross each other, which implies that Theorem \ref{th3.17} does not hold.
\end{counterexample}

The following counterexample shows that the result in Theorem \ref{th3.18} does not hold if we ignore the assumption on the generators and $\boldsymbol{\theta}$ does not belong to $\mathcal{I}_{+}~(or~\mathcal{D}_{+})$.

\begin{counterexample}\label{cex3.2}
	For the baseline distribution as in Counterexample \ref{cex3.1}, let $\{X_{1},X_{2},X_{3}\}$ and
	$\{Y_{1},Y_{2},Y_{3}\}$ be two sets of dependent random variables such that $X_{i}\sim	 F_{b}^{\alpha}(\frac{x-\lambda}{\theta_i})$ and $Y_{i}\sim
	F_{b}^{\alpha}(\frac{x-\lambda}{\delta_i})$, for $i=1,2,3$ with generator $\psi_1(x)=e^{-x^\frac{1}{a_1}}$ and $\psi_2(x)=e^{-x^\frac{1}{a_2}},$ respectively, where $x>0,~a_1,~a_2\geq1$. It is not difficult to check that $\psi_1$ and $\psi_2$ both are log-convex.  Also, for $a_2\leq a_1,$ $\phi_2\circ\psi_1(x)=x^{\frac{a_2}{a_1}}$ is concave, implies $\phi_2\circ\psi_1$ is not super-additive.  Consider $\boldsymbol{\delta}=(4.5,6.5,7.5)$,
	$\boldsymbol{\theta}=(2.5,6.5,3.1)$,
	$\lambda=5,~ \alpha=0.1,~a=0.5,~a_2=1.0001$ and $a_1=2.5$. Here, $\boldsymbol{\theta}$ does not belong to $\mathcal{I}_+ ~(\text{or }\mathcal{D}_+)$ and $\boldsymbol{\delta}\in\mathcal{I}_+.$ Hence,  all the conditions of Theorem \ref{th3.18} except the log-concavity of the generators, super-additivity of $\phi_2\circ\psi_1$ and $\boldsymbol{\theta}\in\mathcal{I}_{+}~(or~\mathcal{D}_{+})$ are satisfied. Next,
	we plot the graphs of $\bar{F}_{X_{2:3}}(x)$ and $\bar{F}_{Y_{2:3}}(x)$ in Figure $2b$, which show that Theorem \ref{th3.18} does not hold.
\end{counterexample}
\begin{figure}[h]
	\begin{center}
		\subfigure[]{\label{c1}\includegraphics[height=2.41in]{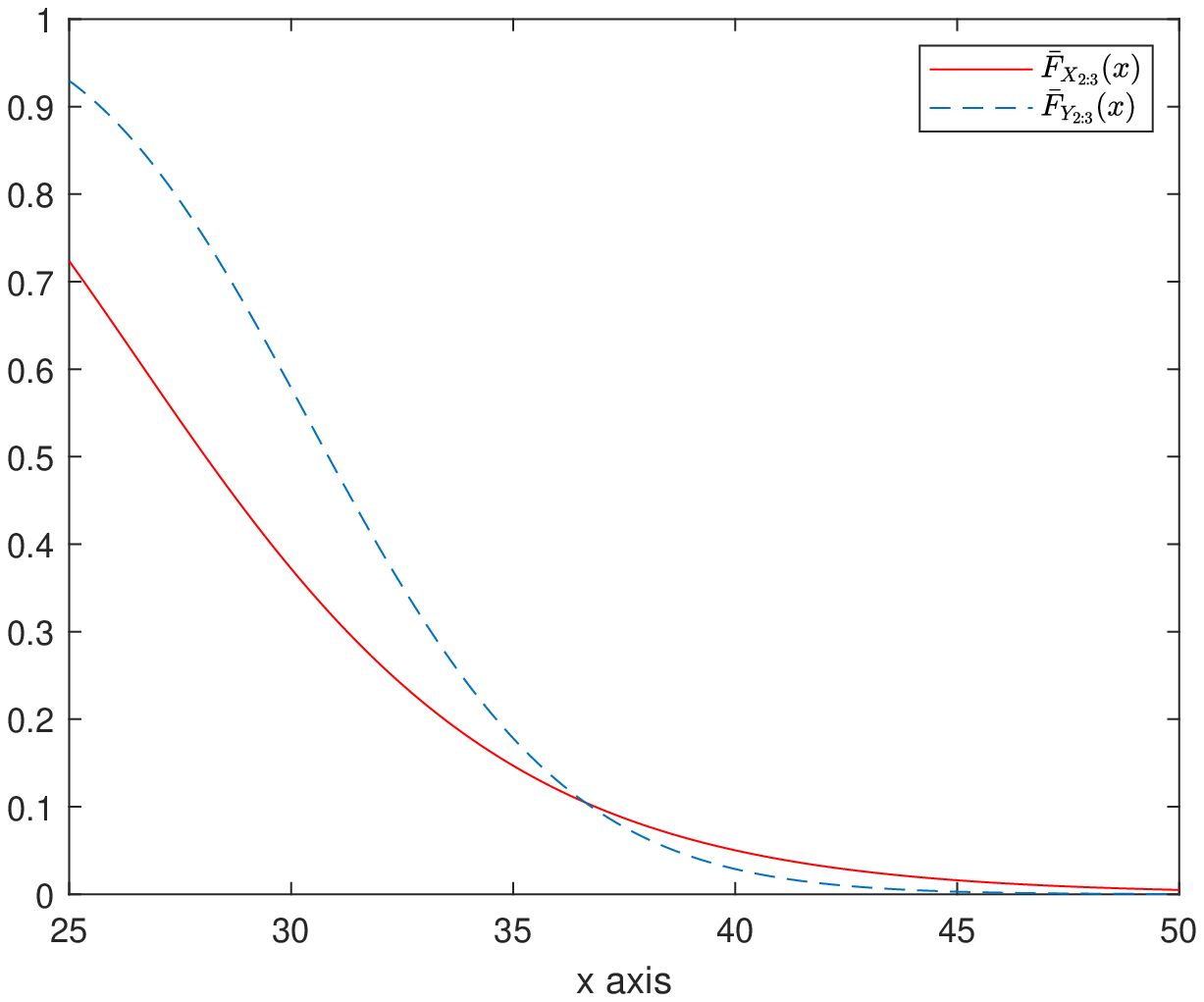}}
		\subfigure[]{\label{c2}\includegraphics[height=2.41in]{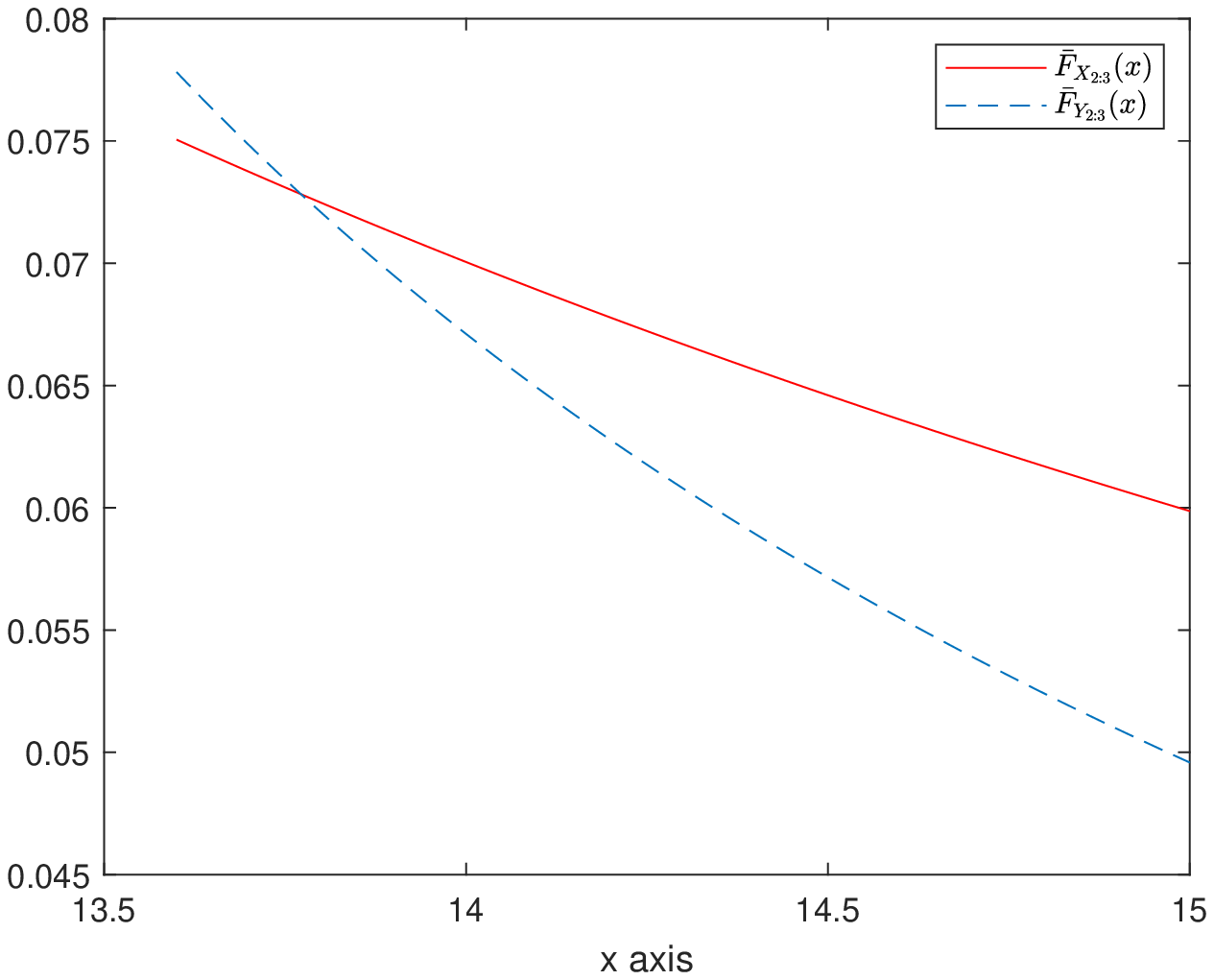}}
		\caption{
			(a) Plots of the survival functions $\bar{F}_{X_{2:3}}(x)$ and $\bar{F}_{Y_{2:3}}(x)$ as in Counterexample \ref{cex3.1}.
			(b) Plots of $\bar{F}_{X_{2:3}}(x)$ and $\bar{F}_{Y_{2:3}}(x)$ as in Counterexample \ref{cex3.2}.}
	\end{center}
\end{figure}

\section{Applications \setcounter{equation}{0}}
In various practical fields such as reliability theory, auction theory and multivariate statistics, the study of ordering results of order statistics is of great importance. In this section, we discuss some applications of the established theoretical results. There are many important fault tolerant structures, which have wide applications in the fields of military system and industrial engineering. Among these, $(n-1)$-out-of-$n$ system has drawn a considerable attention of various researchers. This is known as the fail-safe system.
\begin{itemize}
\item Let us consider two fail-safe systems with $n$ independent components following exponentiated location-scale models. According to Theorem \ref{th3.1}, more heterogeneous location parameters in the weakly submajorized order produces a fail-safe system with stochastically longer lifetime. Similar observations can be noticed from Theorems \ref{th3.2} and \ref{th3.3} for more heterogeneous reciprocal of scale parameters in terms of the weakly supermajorized and reciprocally majorized orders.

\item Due to various factors (working in the same environment, production by same company), components of a system are usually dependent. We assume that there are two fail-safe systems whose components' lifetimes are coupled by Archimedean copula. Under this situation, Theorem \ref{th3.8*} reveals that more heterogeneous location parameters in terms of the weakly submajorized order yield a stochastically longer lifetime of a fail-safe system. Similar interpretation follows from Theorems \ref{th3.9} and \ref{th3.14}, when we have more heterogeneous reciprocal of scale parameters in the sense of the reciprocally majorized and weakly supermajorized orders, respectively. Further, from Theorems \ref{th3.15}$(ii)$ and \ref{th3.16}$(ii)$, we observe that the less positive dependence, and the more heterogeneous location and reciprocal of scale parameters, respectively in the sense of the weakly submajorized and weakly supermajorized orders produce systems with greater reliability.

\item In auction theory, the auctioneers are always interested to understand the impact of the dependence among bids. Few of the established results could be useful in this field. We assume that the bids follow exponentiated location-scale model and coupled by Archimedean copulas. Then, Theorems \ref{th3.15}$(ii)$ and \ref{th3.16}$(ii)$ state that in the second-price reverse auction, the less positive dependence, and the more heterogeneous bids will be stochastically larger. These results are helpful to the auctioneers in order to realize valuable information, since the dependence and heterogeneity on bids strongly associate with the information. This may effect the said price very badly.
\end{itemize}

\section{Conclusion\setcounter{equation}{0}}
In this paper, we studied ordering results for the second-order statistics arising from two sets of heterogeneous random observations. The random variables follow exponentiated location-scale distributions. Here, we have considered two cases: independent and interdependent random observations. For the case of independent observations, we studied comparisons of the second-order statistics in terms of the usual stochastic order and the hazard rate order. Similar results have been obtained for the case of the dependent random variables under different sets of conditions. In this case, we assumed that the groups of dependent observations share a common Archimedean copula. Few results have also been derived for Archimedean copula with different generators. The established results are useful to derive bounds of the time to failure of  fail-safe systems with heterogeneous components in terms of the fail-safe systems with identical components. These results are also useful in improvement of the reliability of fail-safe systems and in understanding the impact of dependence among the bids in second-price reverse auction.
\\
\\
\section*{Acknowledgements:}
Sangita Das wishes to thank the Ministry of Education (formerly known as MHRD), Government of India for the financial support to carry out this research work. Suchandan Kayal gratefully acknowledges the partial financial support for this research work under a grant MTR/2018/000350, SERB, India.
\section*{Disclosure statement:}
Both the authors states that there is no conflict of interest.

\bibliography{ref}
\end{document}